\documentclass[11pt]{amsart}
\usepackage{mathdots}
\usepackage{hyperref}
\usepackage{amsmath,amsthm,amssymb} 
\usepackage[numbers]{natbib} %%%----for ELSARTICLE-----------
\usepackage{tabularx}
\usepackage{enumerate}
\theoremstyle{plain}
\newtheorem{thm}{Theorem}[section]
\newtheorem{prop}[thm]{Proposition}
\newtheorem{cor}[thm]{Corollary}
\numberwithin{equation}{section}
\newtheorem{lemma}[equation]{Lemma}
\newtheorem{theorem}[equation]{Theorem}
\newtheorem{utheorem}{\textrm{\textbf{Theorem}}}

\theoremstyle{definition}
\newtheorem{quest}[thm]{Question}
\newtheorem{example}[thm]{Example}
\newtheorem{defn}[thm]{Definition}
\newtheorem{notn}[thm]{Notation}
\newtheorem{remark}[thm]{Remark}
\newcommand{\diag}{\mathrm{diag}}
\newcommand{\F}{\mathbb{F}}
\newcommand{\C}{\mathbb{C}}
\newcommand{\R}{\mathbb{R}}
\newcommand{\N}{\mathbb{N}}
\newcommand{\Z}{\mathbb{Z}}

\newcommand{\I}{\mathrm{I}}

\newcommand{\Pn}{\mathbb{P}}
\newcommand{\GL}{\mathrm{GL}}
\newcommand{\jury}{\diamond}

\title[Functional Calculi, Positivity, and Convolution of Matrices]{Functional Calculi, Positivity, and Convolution of Matrices}

\author[J.~Mashreghi, M.~Nasri, and P.K.~Vishwakarma]{Javad Mashreghi, Mostafa Nasri,\\ and Prateek Kumar Vishwakarma}

\address[Javad Mashreghi]{D\'epartement de math\'ematiques et de statistique, Universit\'e Laval, Qu\'ebec, QC, Canada G1V 0K6.}
\email{javad.mashreghi@ulaval.ca}

\address[Mostafa Nasri]{Department of Mathematics and Statistics, University of Winnipeg, Winnipeg, MB, Canada R3B 2E9}
\email{m.nasri@uwinnipeg.ca}

\address[Prateek Kumar Vishwakarma]{D\'epartement de math\'ematiques et de statistique, Universit\'e Laval, Qu\'ebec, QC, Canada G1V 0K6.}
\email{prateek-kumar.vishwakarma.1@ulaval.ca,~prateekv@alum.iisc.ac.in}

\date{\today}
\keywords{Convolution, operator monotone functions, functional calculus, matrix transform, positivity preserver, Bruhat order}

\subjclass[2020]{47B49, 47C05, 15B48}

\thanks{This work was supported by the Alliance and Discovery Grants of NSERC (Canada), and the Canada Research Chairs program.}

\begin{document}
\begin{abstract}
Convolution admits a natural formulation as a functional operation on matrices. Motivated by the functional and entrywise calculi, this leads to a framework in which convolution defines a matrix transform that preserves positivity. Within this setting, we establish results parallel to the classical theories of Pólya--Szegő, Schoenberg, Rudin, Loewner, and Horn in the context of entrywise calculus. The structure of our transform is governed by a Cayley--Hamilton-type theory valid in commutative rings of characteristic zero, together with a novel polynomial-matrix identity specific to convolution. Beyond these analytic aspects, we uncover an intrinsic connection between convolution and the Bruhat order on the symmetric group, illuminating the combinatorial aspect of this functional operation. This work extends the classical theory of entrywise positivity preservers and operator monotone functions into the convolutional setting.
\end{abstract}

\maketitle

\section{Introduction and historical background}

Functions on matrix spaces are called matrix transformations. They are pervasive across many areas of pure mathematics and the applied sciences. Of particular interest are those that preserve positivity, specifically the cone of positive semidefinite matrices. A transformation is called a positivity preserver if it maps positive semidefinite matrices into itself, thereby maintaining the spectral properties crucial in many fundamental, theoretical, and applied contexts. In this work, we look at convolution as a multiplication of matrices, formulate the corresponding matrix transforms, and classify the relevant positivity preservers. We begin with two classical motivations.

\subsection{Entrywise calculus}

Among the most extensively studied positivity preserving maps are the entrywise transforms. These are scalar functions extended to matrix spaces by operating on each matrix entry individually. The foundation of this theory dates back to the Schur Product Theorem \cite{Schur1911}: the entrywise product of positive semidefinite matrices (of the same size) is again positive semidefinte. This fundamental result immediately raises the natural question: which scalar functions preserve positivity when applied entrywise? Early insights into this question were provided by P\'olya and Szeg\H{o} \cite{polyaszego}, who observed that a convergent power series with nonnegative Maclaurin coefficients (also called absolutely monotonic) when applied entrywise, preserve positivity. With applications to metric geometry on the Hilbert sphere, the theory was significantly advanced by Schoenberg, who gave a complete characterization of these for continuous functions. This saw improvements in later work of Rudin (and several others in the recent times due to motivations from high-dimensional covariance estimations).

\begin{theorem}[Schoenberg \cite{Schur1911}, Rudin \cite{Rudin-Duke59}]\label{Tschoenberg}
Consider the interval $I=(-1,1)$ and a function $f:I\to\R$. The following are equivalent.
\begin{enumerate}[$(1)$]
    \item The matrix $f[A]:=(f(a_{ij}))\in \R^{n\times n}$ is positive semidefinite whenever $A=(a_{ij})\in I^{n\times n}$ is positive semidefinite, for all intergers $n\geq 1$.
    \item The function $f(x)\equiv \sum_{k=0}^{\infty}c_{k}x^{k}$ on $I$, where each $c_{k}\geq 0$.
\end{enumerate}
\end{theorem}

While Schoenberg's theorem applies to functions preserving positivity for matrices of any size, characterizing entrywise positivity preservers for matrices of a fixed dimension remains a difficult, largely unresolved problem (except for the dimension $2$ case resolved by Vasudeva \cite{vasudeva1979positive}). Horn \cite{horn1969theory} provided a crucial necessary condition that such functions must exhibit a certain degree of smoothness, with nonnegative derivatives. In a breakthrough publication, Belton, Guillot, Khare, and Putinar \cite{belton2016matrix} solved the problem for polynomials of degree at most $N$, showing how they preserve positivity on $N\times N$ matrices (and introduced the first known example of a non-absolutely monotone polynomial that preserves positivity in fixed dimensions). Khare and Tao \cite{khare2021sign} further explored the structure of these functions with connections to symmetric function theory. Additionally, several variants have also been explored, such as those involving structured matrices \cite{guillot2016critical,guillot2016preserving,belton2021moment}, specific functions \cite{fitzgerald1977fractional, guillot2015complete,guillot2016critical}, block actions \cite{vishwakarma2023positivity,guillot2015functions}, an algebraic setting of finite fields \cite{guillot2025positivity,vishwakarma2025cholesky,GuillotGuptaVishwakarmaYip2025FPSAC}, settings of inertia preservers \cite{belton2023negativity,GuillotGuptaVishwakarmaYip2025}, a noncommutative version \cite{pascoe2019noncommutative}, and several others. One may refer to the monograph by Khare \cite{khare2022matrix} for a more detailed exposition.

\subsection{Functional calculus} While entrywise transforms interact with matrices element-wise, another classical setting in which positivity preservers arise is the functional calculus for Hermitian matrices. In this framework, a function $f:\R\to\R$ is applied to a Hermitian matrix $A$ via its orthogonal spectral decomposition: if $A=UDU^*$ with $D=\diag(d_1,\dots,d_n)$, then $f(A):=Uf(D)U^*$, where $f(D):=\diag(f(d_1),\dots,f(d_n))$. In this setting, therefore, the positivity preservers refer to $f$ which are nonnegative on nonnegative reals. A more nuanced class of functions emerges when we restrict the requirement from positivity to monotonicity, which is the next ideal approach in the central concept of preserving the Loewner order. This leads to the well-studied operator monotone functions of Loewner. For Hermitian matrices $A,B$ of the same size, we say $A\leq B$ in the Loewner order $\leq$ if $B-A$ is positive semidefinite.

\begin{theorem}[Loewner \cite{lowner1934monotone}] Let $(a,b)\subseteq \R$ be an interval, and suppose $f:(a,b)\to \R$ is a function. Then the following are equivalent.
\begin{enumerate}[$(1)$]
    \item $f$ is operator monotone: $f(A)\leq f(B)$ for all Hermitian $A\leq B\in \C^{n\times n}$ with eigenvalues in $(a,b)$, for all integers $n\geq 1$.
    \item $f\in C^1((a,b))$, and the Loewner matrix $L_f\in \R^{k\times k}$ given by
    \begin{align*}
        (L_f)_{ij}:=
        \begin{cases}
        \frac{f(x_i)-f(x_j)}{x_i-x_j} & \mbox{if $i\neq j$ and}\\
        f'(x_i) & \mbox{otherwise},
        \end{cases}
    \end{align*}
    for all $a<x_1<\dots<x_k<b$, and all $k\geq 1$, is positive semidefinite.
\end{enumerate}
\end{theorem}

This characterization is equivalent to the function \( f \) admitting an analytic continuation to the upper half-plane that maps it into itself. Equivalently, \( f \) has an analytic extension to the domain \( (\mathbb{C} \setminus \mathbb{R}) \cup (a, b) \), which maps the upper half-plane into itself, with the extension to the lower half-plane obtained via reflection across the real axis. The theory of operator monotone functions has found numerous important applications~\cite{wigner1993significance, kubo1980means}, and over the years, this profound result has attracted significant interest. A multivariable generalization was recently developed by Agler, McCarthy, and Young \cite{agler2012operator}. For a detailed treatment of the classical Loewner theorem, see Donoghue’s book~\cite{donoghue2012monotone}, and for several alternative proofs, refer to Simon’s book~\cite{simon2019loewner}.\medskip

We just presented only a glimpse of a broader body of work on functional and entrywise calculi, within which positivity preserving transforms form a particularly rich and diverse area of study. At the heart of these developments lies the fundamental question of how scalar functions can \textit{meaningfully} be extended to matrix spaces in a manner that \textit{respects} the standard and entrywise multiplication of matrices. Building on this perspective, we introduce a new line of inquiry.

In this work, we focus on a matrix multiplication defined via convolution (Definition~\ref{defn:jury-prod}). Convolution of matrices, or matrix convolution, has deep classical roots in mathematics, arising in harmonic analysis of one variables, in integral transforms and operational calculus, in integral equations of the convolution type (Abel, Picard, Toeplitz and Wiener--Hopf type), probability theory, and the study of convolution algebras. At the same time, it also plays a central role in applied domains such as image and audio processing, deep learning, and numerical simulations, etc. A less widely recognized but crucial fact is that, much like the Schur Product Theorem, the convolution of positive semidefinite matrices is positivity semidefinite (Theorem~\ref{T:jury}).\medskip

\section{An overview of our main contributions}
In light of the above introduction and the explanations at the end, we are naturally led to the following fundamental questions: What is the appropriate analogue of functional and entrywise calculi for convolution of matrices? More precisely, how can a scalar function be extended to matrix spaces in a way that is compatible with convolution? Under what conditions does such an extension preserve matrix positivity?\medskip

These questions form the core motivation for our work and lead directly to the main contributions of this paper. Since the detailed discussion is technical (and natural), it is developed progressively in Sections~\ref{S:jury+comb+prob}, \ref{S:jury-prod-mat-trans}, \ref{S:pos-preservers}, and \ref{S:cayley-hamilton-thm}. We present {\bf seven main results}, stated as {\bf Theorems \ref{T:polya-szego-for-jury} through \ref{T:poly-on-mat}}, which we now outline.\medskip

\noindent \textit{Convolution as positivity preservers:} Viewing convolution as a matrix transform, we extend the theory of Schoenberg’s entrywise positivity preservers and Loewner’s operator monotone functions. The main results in this direction are the following.
\begin{description}
    \item[Theorem~\ref{T:polya-szego-for-jury}] is inspired by P\'olya and Szeg\H{o} \cite{polyaszego}, whose key observation led to entrywise calculus and positivity preservers. See Page \pageref{T:polya-szego-for-jury}.
    \item[Theorem~\ref{T:schoeberg-for-jury}] is in the spirit of Schoenberg’s Theorem~\ref{Tschoenberg}, which answered P\'olya and Szeg\H{o}'s question \cite{polyaszego} on absolute monotonicity. See Page \pageref{T:schoeberg-for-jury}.
    \item[Theorem~\ref{T:Rudin-for-jury}] parallels Rudin’s refinement \cite{Rudin-Duke59} of Schoenberg’s original proof \cite{schoenberg1942positive} given only for continuous functions. See Page \pageref{T:Rudin-for-jury}.
    \item[Theorem~\ref{T:horn-for-jury}] is an analogue of Horn’s theorem \cite{horn1969theory} in the harder fixed dimensional setting that refined Schoenberg’s theorem on absolute monotonicity. See Page \pageref{T:horn-for-jury}.
    \item[Theorem~\ref{T:fh-for-jury}] extends the fractional Hadamard powers of FitzGerald and Horn \cite{fitzgerald1977fractional} to convolution of matrices. See Page \pageref{T:fh-for-jury}.
\end{description}

\noindent \textit{Convolution as matrix transforms:} The previous part is naturally set against the formulation of the matrix transform that respects convolution of matrices. These transforms are given in Definitions~\ref{jury-transform-1} and \ref{jury-transform-2}, and the following results are crucial for their formulation.
\begin{description}
    \item[Theorem~\ref{main-thm-1}] is a Cayley--Hamilton-type result for convolution. See Page \pageref{main-thm-1}.
    \item[Theorem~\ref{T:poly-on-mat}] is a novel polynomial-matrix identity for convolution. See Page \pageref{T:poly-on-mat}.
\end{description}

\noindent \textit{Bruhat order and discrete probability:} We also study a natural connection between the convolution and the Bruhat order in Subsection~\ref{S:jury+bruhat}. Then, inspired by the probabilistic connection of the Schur Product Theorem (via a proof in {\cite{khare2022matrix}}), we show that the convolution yields a positivity-type implication on sums of discrete random variables in Subsection~\ref{S:jury+prob}.\medskip

Hence, in the end, by presenting convolution with the point of view of positivity and matrix algebra, we aim to stimulate further research and developments in structured, high-dimensional matrix analysis, and possibly with applications to combinatorics and discrete probability.

\subsection{Organization of the paper} This article investigates several properties of convolution of matrices, and is structured as follows. In Section~\ref{S:jury+comb+prob} we introduce it and present its properties; then discuss some key differences of it with the usual and the Schur (entrywise) product; then present a novel connection with the Bruhat order, and show an immediate application in probability. In Section~\ref{S:jury-prod-mat-trans}, we define the novel matrix transforms associated with convolution. Building on this, in Section~\ref{S:pos-preservers} we present the positivity preservers for our matrix transforms. Next, in Section~\ref{S:cayley-hamilton-thm} we show a Cayley--Hamilton-type theorem underlying the novel matrix transforms. In the penultimate Section~\ref{S:proof-pos-preservers} we present the proofs for our positivity preservers. Inspired by the open questions from this work, finally in Section~\ref{S:concluding} we discuss a qualitative interpretation of our transforms for power functions.

\section{Convolution of matrices and its properties}\label{S:jury+comb+prob}

We begin with a formal definition.

\begin{defn}[The matrix convolution]\label{defn:jury-prod}
Henceforth we use the following notations.
\begin{enumerate}[$(a)$]
    \item Define $[r:s]:=[r,s]\cap \mathbb{Z}$ for all $r\leq s\in \mathbb{R}$, and $[r:s]:=\emptyset$ if $r>s$. Let
    \[
    \mathcal{K}_{r,s}:=[0:r-1]\times [0:s-1] \quad \mbox{and}\quad \mathcal{K}_{r,s}^*:=\mathcal{K}_{r,s}\setminus\{(0,0)\}
    \]
    for all integers $r,s\geq 1$. (The notation $[r:s]$ is adopted from \cite{belton2023negativity}.)
    \item For integers $M,N\geq 1$, the convolution of matrices $A=(a_{ij})$ and $B=(b_{ij})\in \mathbb{C}^{M\times N}$, indexed by $(i,j)\in \mathcal{K}_{M\times N}$, is denoted by $A\jury B \in \mathbb{C}^{M\times N}$, and is defined by
\begin{equation*}
(A\jury B)_{ij}:=\sum_{\ell = 0}^{i} \sum_{\kappa=0}^{j} a_{\ell,\kappa}b_{i-\ell,j-\kappa} \quad \mbox{for all } (i,j)\in \mathcal{K}_{M,N}.
\end{equation*}
\end{enumerate}
\end{defn}
Some of the algebraic properties of matrix convolution may follow from the more general theory on convolution algebras. But for completeness we mention them here, purely in the setting of matrix spaces (which we later use in our proofs). The matrix convolution is associative, commutative, and compatible with scalar multiplication, matrix addition, and transposition:
\begin{eqnarray}\label{E:jury-property}
A\jury B=B\jury A, &&~ (A \jury B)^T = A^T \jury B^T,\\
(A\jury B)\jury C = A\jury (B\jury C),&&~ (A + B)\jury C = A\jury C + B\jury C,\nonumber\\
\alpha(A \jury B) = (\alpha A) \jury B = A \jury (\alpha B),&&~  \mbox{for all }A,B,C\in \C^{M\times N}, ~ \alpha\in \C.\nonumber
\end{eqnarray}
Moreover, it admits the following multiplicative identity element:
\begin{align}\label{E:jury-identity}
\I_{\jury}:=\begin{pmatrix}1\end{pmatrix}\oplus {\bf 0}_{(M-1)\times (N-1)} \in \C^{M\times N}.
\end{align}
Furthermore, matrix convolution admits mutliplicative inverses (which we state without proof): a matrix $A:=(a_{ij})\in (\C^{M\times N},\jury,+)$ has a multiplicative inverse if and only if $a_{00}\neq 0$. In which case, the inverse $A^{\jury -1}:=B=(b_{ij})\in \C^{M\times N}$ is unique, and is defined by:
\begin{align}\label{E:jury-inverse}
\begin{aligned}
b_{00}&:=a_{00}^{-1},\quad\mbox{and}\\
b_{ij}&:= -a_{00}^{-1} \Big{(} (A\jury B)_{ij} - a_{00}b_{ij} \Big{)}, \quad \mbox{for all } (i,j)\in \mathcal{K}_{M,N}^*,
\end{aligned}
\end{align}
such that all $b_{\ell,\kappa}$ for $(\ell,\kappa)\in \mathcal{K}_{i+1,j+1}\setminus \{(i,j)\}$ are defined \textit{before} defining $b_{ij}$. One may have noticed that this formulation of the inverse is inductive/algorithmic. We refer to Remark~\ref{rem:ch-thm-1} for its direct computation (as an application of our Cayley--Hamilton-type Theorem~\ref{main-thm-1}). These show that if $A,B\in \C^{M\times N}$ are invertible then $A\jury B$ is so, and moreover:
\begin{align}\label{E:jury-inverse-2}
(A\jury B)^{\jury -1} = A^{\jury -1}\jury B^{\jury -1}.
\end{align}

In summary, \(\mathbb{C}^{M \times N}\) forms a unital commutative ring under convolution together with the usual matrix addition. More generally, the same holds when \(\mathbb{C}\) is replaced by any unital commutative ring, however the existence of multiplicative inverses requires a field structure. Beyond these algebraic properties, matrix convolution also preserves matrix positivity.

\begin{defn}[Positive semidefinite matrix] For an integer $N\geq 1$, a matrix $A\in \C^{N\times N}$ with $A^*=A$ is called positive semidefinite if $z^*Az\geq 0$ for $z\in \C^N$. Moreover, $A$ is called positive definite if $z^*Az> 0$ for $z\in \C^N\setminus\{0\}$.
\end{defn}

\begin{theorem}[Jury \cite{jurythesis}]\label{T:jury}
Suppose $N\geq 1$ is an integer, and \(A,B\in\C^{N\times N}\) are positive (semi)definite. Then \( A \jury B \) is positive (semi)definite.
\end{theorem}

This result was proved in the 2002 Ph.D.\ thesis of M.~Jury, and developed as part of a unified approach towards classical Carathéodory--Fejér and Nevanlinna--Pick interpolation problems in complex function theory. In Jury’s approach, these problems are formulated as extension problems for positive matrices. For example, the Carathéodory interpolation problem~\cite{Car1907}, later studied by Schur~\cite{schur1917,schur1918}, asks for a holomorphic function on the unit disc whose Taylor expansion begins with a prescribed sequence and whose modulus remains bounded by one throughout the disc. To address such questions, Jury introduced a product defined relative to a partial order on a fiber bundle and employed it to solve the classical interpolation problems. His construction further provides a foundation for studying generalized versions of these problems. See the book of Agler and McCarthy~\cite{agler2002pick} for more background.

\begin{remark}[Recognition \& Nomenclature]
The motivation for developing an introduction to positivity preservers for matrix convolution -- in parallel with the well-established theory for the Schur (entrywise) product -- arises from Theorem~\ref{T:jury}. Proven more than two decades ago, this theorem, like the Schur Product Theorem \cite{Schur1911}, has its roots in applications to complex function theory. In recognition of its foundational role, we shall henceforth also refer to matrix convolution (or the convolution of matrices) as the \textit{Jury product}, and we will use these terms interchangeably.
\end{remark}

\subsection{Jury's product versus the standard and Schur products. I}\label{Sub:jury-v-schur}
A main distinction between Jury's product and the standard and Schur products is how it interacts with block diagonal matrices. Both, the standard and Schur products, preserve block diagonal structure, whereas the Jury product does not. More precisely, one can find matrices $A,B,P,Q$ of compatible dimensions such that
\[
\begin{pmatrix}
    A & {\bf 0}\\ {\bf 0} & B
\end{pmatrix} \jury
\begin{pmatrix}
    P & {\bf 0}\\ {\bf 0} & Q
\end{pmatrix} \;\neq\;
\begin{pmatrix}
    A\jury P & {\bf 0}\\ {\bf 0} & B\jury Q
\end{pmatrix}.
\]

Mainly as a consequence of the above, the Jury product also possesses another distinguishing property, arising from its Cayley–Hamilton theory, later shown in Subsection~\ref{R:CH-for-unbound-mat}. Such differences may underlie its surprising connections: on one hand, to a classical notion in algebraic geometry via a well-studied and widely applicable concept in combinatorics; and on the other, to a standard probabilistic framework for discrete random variables. We briefly explore these next.

\subsection{Jury's product and the Bruhat decomposition}\label{S:jury+bruhat}

For an integer $n \geq 1$, the seminal \emph{Bruhat decomposition} expresses the general linear group as a disjoint union of double cosets:
\[
\GL_n(\C) = \bigsqcup_{P \in \mathrm{P}_n} B P B,
\]
where $\mathrm{P}_n \subseteq \GL_n(\C)$ is the subgroup of permutation matrices (the Weyl group), and $B \subseteq \GL_n(\C)$ is the subgroup of upper triangular matrices (the Borel subgroup). Most notably, this yields a decomposition of the flag veriety. The \emph{full flag variety} of $\C^n$ is the set of complete flags
\(
\mathcal{F}_n(\C) := \{ F^{\bullet} : 0 = F^0 \subset F^1 \subset \cdots \subset F^n = \C^n, \ \dim F^i = i \}
\). It admits $\mathcal{F}_n(\C) \cong \GL_n(\C)/B$, and the Bruhat decomposition of $\GL_n(\C)$ induces a decomposition of the flag variety:
\[
\mathcal{F}_n(\C) = \bigsqcup_{P \in \mathrm{P}_n} BPB/B, \quad \mbox{where each }  C_{\omega(P)} := BPB/B
\]
is a \emph{Bruhat cell} (or \emph{Schubert cell}), indexed by the permutation $\omega(P) \in S_n$ corresponding to $P$. Moreover, a striking fundamental fact is that the Zariski closure of a Bruhat cell is a union of smaller cells, governed by the \emph{Bruhat order} $\leq_{\mathrm{Bruhat}}$ on permutations:
\[
\overline{C_{\omega(P)}} = \bigsqcup_{\substack{Q \in \mathrm{P}_n: \\ \omega(Q) \leq_{\mathrm{Bruhat}} \omega(P)}} C_{\omega(Q)}.
\]

The Bruhat order, originally defined via closure relations among double cosets $BwB$ in reductive groups, provides the combinatorial meaning for how Schubert cells fit together in flag varieties. The same structure extends from the symmetric group to Weyl groups and, more generally, to arbitrary Coxeter groups, where it admits the algebraic \emph{subword property}. In the case of the symmetric group, the Bruhat order also admits a combinatorial description in terms of \emph{Young tableaux}, offering an alternative and more visual perspective on the same inclusion relations among Schubert cells. These connections were developed in the classical works \cite{Chevalley1955, Bruhat1956, Tits1960, Bourbaki1968}, and received modern treatments in \cite{BjornerBrenti2005, Humphreys1990}.

Adding to the discourse above, recent works have approached the Bruhat order through notions of positivity. For instance, the main theorem of \cite{drake2005monomial} shows that permutations $\omega \leq_{\mathrm{Bruhat}} \mu$ if and only if the difference of monomials, which are also sometimes called generalized diagonals,
\[
x_{1,\omega(1)} \cdots x_{n,\omega(n)} - x_{1,\mu(1)} \cdots x_{n,\mu(n)}
\]
satisfies a variety of positivity conditions, where ${\bf x}=(x_{ij})_{i,j=1}^n$ is a matrix of indeterminates. This reveals an interplay between the combinatorial Bruhat order and the algebraic notions of positivity, with applications in combinatorics, representation theory, and linear algebra \cite{drake2004two, drake2005monomial}.

In the context above -- and noting that determining comparability in the Bruhat order is often tedious -- we present a direct connection between the Jury product (the matrix convolution) and the Bruhat order on the symmetric group.

\begin{theorem}\label{T:Jury+Bruhat}
Let $n \geq 1$, and suppose for permutation matrices $P,Q \in \GL_n(\C)$, we use $\omega(P),\omega(Q)\in \mathrm{S}_n$ to denote the respective permutations. Then all of the following relations are equivalent.
\begin{align*}
\omega(Q) \leq_{\mathrm{Bruhat}} \omega(P) ~\ ~
& P\jury \mathbf{1}_{n\times n} \leq_{\mathrm{Entry}} Q \jury \mathbf{1}_{n\times n} \\
\omega(\nabla_n) \omega(P) \leq_{\mathrm{Bruhat}} \omega(\nabla_n) \omega(Q) ~\ ~& (\nabla_n Q)\jury \mathbf{1}_{n\times n} \leq_{\mathrm{Entry}} (\nabla_n  P) \jury \mathbf{1}_{n\times n}\\
\omega(P)\omega(\nabla_n) \leq_{\mathrm{Bruhat}} \omega(Q)\omega(\nabla_n) ~\ ~& (Q  \nabla_n)\jury \mathbf{1}_{n\times n} \leq_{\mathrm{Entry}} (P  \nabla_n) \jury \mathbf{1}_{n\times n} \\
\omega(Q)^{-1} \leq_{\mathrm{Bruhat}} \omega(P)^{-1} ~\ ~ & P^T\jury \mathbf{1}_{n\times n}\leq_{\mathrm{Entry}}Q^T \jury \mathbf{1}_{n\times n}
\end{align*}
Here $\nabla_n$ is the $n\times n$ anti-diagonal permutation matrix (that reverses the rows of $P,Q$); $\mathbf{1}_{n\times n}$ is the $n \times n$ matrix whose entries are all equal to $1$;  $\leq_{\mathrm{Entry}}$ denotes the partial order defined by entrywise comparison of matrices.
\end{theorem}

\begin{remark}[About our proof]
The equivalent partial order relations in Theorem~\ref{T:Jury+Bruhat} are presented in two columns: on the left with permutations $\omega(P),\omega(Q)$, and on the right with the corresponding permutation matrices $P,Q$. The equivalence on the permutation side is well known from the theory of automorphisms of the Bruhat graphs and digraphs, as $\omega(\nabla_n)$ refers to the \textit{longest element} $n~(n-1)~\dots~2~1 \in \mathrm{S}_n$. What remains and what we establish here are the remaining equivalences using properties of the Jury product and permutation matrices.
\end{remark}

\begin{proof}[Proof of Theorem~\ref{T:Jury+Bruhat}]

First, the equivalence between $\omega(Q) \leq_{\mathrm{Bruhat}} \omega(P)$ and $P\jury \mathbf{1}_{n\times n} \leq_{\mathrm{Entry}} Q \jury \mathbf{1}_{n\times n}$ can be deduced from {\cite[Theorem~1]{drake2005monomial}} (and more details can be found in {\cite[Theorem~2.1.5]{BjornerBrenti2005}} and {\cite[Section 10.5]{Fulton1997}}).

Next, it is not difficult to see that the following holds for all permutation matrices $A\in \GL_n(\C)$ indexed by $i,j\in \mathcal{K}_{n,n}$:
\begin{align}\label{E:T:bruhat-1}
j+1=(A\jury \mathbf{1}_{n\times n})_{n-1,j}=\big{(}A\jury {\bf 1}_{n\times n}\big{)}_{i,j} + \big{(}(\nabla_n A)\jury {\bf 1}_{n\times n}\big{)}_{n-1-i,j}.
\end{align}
Then using the above, switching $A$ by $A^T$, and $i$ and $j$, we get:
\begin{align*}
i+1&=\big{(}A^T\jury {\bf 1}_{n\times n}\big{)}_{j,i} + \big{(}(\nabla_n A^T)\jury {\bf 1}_{n\times n}\big{)}_{n-1-j,i}\\
&=\big{(}A^T\jury {\bf 1}_{n\times n}^T\big{)}_{j,i} + \big{(}(A\nabla_n)^T\jury {\bf 1}_{n\times n}^T\big{)}_{n-1-j,i}.
\end{align*}
The above, combined with the properties of the Jury product \eqref{E:jury-property}, gives:
\begin{align}\label{E:T:bruhat-2}
i+1&=\big{(}A\jury {\bf 1}_{n\times n}\big{)}^T_{j,i} + \big{(}(A\nabla_n)\jury {\bf 1}_{n\times n}\big{)}^T_{n-1-j,i}\nonumber \\
&=\big{(}A\jury {\bf 1}_{n\times n}\big{)}_{i,j} + \big{(}(A \nabla_n)\jury {\bf 1}_{n\times n}\big{)}_{i,n-1-j}.
\end{align}

Then, to obtain the equivalence of the first three statements on the right side, write identities \eqref{E:T:bruhat-1} and \eqref{E:T:bruhat-2} for $P,Q$ and compare the entries.

Then, for the final equivalence, we use a property \eqref{E:jury-property} of the Jury product that
\((A\jury B)^T = A^T \jury B^T\), together with the observation that the linear map \(A \mapsto A^T\) preserves the entrywise order.
\end{proof}

We note that this brief discussion of the matricial and combinatorial perspectives of the Bruhat order is only the tip of the iceberg, and numerous related works can be found through a simple literature search. For instance, it is natural to ask whether the matrix viewpoint on the Bruhat order extends beyond permutation matrices to broader families, and whether such extensions reveal new phenomena. A recent contribution in this direction is to total positivity, where the authors show that a rational function in \textit{permanents} admits a uniform upper bound over totally positive matrices if and only if a pair of carefully constructed ``index'' matrices -- which need not be permutations -- are comparable in the Bruhat order {\cite[Theorem~3.3]{skandera2025permanental}}.

\subsection{Jury's product and sums of discrete random variables}\label{S:jury+prob} This is inspired from the probabilistic connection of the Schur Product Theorem, via its proof in {\cite[Theorem~2.10]{khare2022matrix}}. Consider a random variable $X$ supported on the two-dimensional integer grid $\Z_{\ge 0} \times \Z_{\ge 0}$. Let $A(X)$ denote its semi-infinite probability matrix, indexed by $\mathcal{K}_{\infty,\infty} := [0:\infty) \times [0:\infty)$, with entries
\[
A(X)_{ij} := \mathrm{Prob}(X = (i,j)) \quad \text{for all } (i,j) \in \mathcal{K}_{\infty,\infty}.
\]
These semi-infinite matrices encode the full probability distribution of $X$, and their Jury product is defined analogously to Definition~\ref{defn:jury-prod}. Moreover, $A(X)$ provides a direct formula, via the Frobenius inner product, for computing both the expectation and the covariance matrix of $X$. While it is well-known that the covariance matrix of any multivariate random variable is positive semidefinite, Theorem~\ref{T:jury} yields an unconventional property: if the probability matrices of random variables $X$ are positive semidefinite, then so are the matrices corresponding to their finite sums.

\begin{cor}\label{C:Jury+prob+pos}
Let $X_1,\dots,X_n$ be independent random variables supported on $\Z_{\ge 0} \times \Z_{\ge 0}$. Then, we know that the probability matrix of their sum is given by convolution:
\[
A(X_1+\cdots+X_n) = A(X_1) \jury \cdots \jury A(X_n).
\]
Hence, by Theorem~\ref{T:jury}, if $A(X_1), \dots, A(X_n)$ are positive semidefinite, then $A(X_1+\cdots+X_n)$ is also positive semidefinite. Here, a semi-infinite matrix is positive semidefinite if all its leading principal submatrices are positive semidefinite.
\end{cor}

\section{Convolution as matrix transforms}\label{S:jury-prod-mat-trans}

If we look back at the induced action of polynomials over matrices for the standard matrix product and the Schur (entrywise) product over $\C^{N\times N}$, then for $p(z)=a_0+a_1z+\cdots+a_nz^n \in \C[z]$, we have
\begin{align*}
    p(A)&:=a_0\I +a_1 A + a_2 A^2 + \dots + a_n A^n, \qquad \mbox{(standard)}\\
    p[A]&:=a_0\I_{\circ} +a_1 A + a_2 A^{\circ 2} + \dots + a_n A^{\circ n}, \qquad \mbox{(Schur)}
\end{align*}
where $\I\in \C^{N\times N}$ is the standard identity matrix, $\I_{\circ}:=(1)_{N\times N}$ is the identity of the Schur product, and $A^{\kappa}=A\cdot\dots\cdot A$ and $A^{\circ \kappa}=A\circ \dots\circ A$ are defined, respectively, by the standard and the Schur product. Therefore, it is reasonable to define the corresponding action of polynomials on every $A\in \C^{M\times N}$ equipped with matrix convolution (the Jury product) by
\begin{equation}\label{defn:p-of-A-for-jury}
p_{\jury}(A) := a_0 \I_{\jury} +a_1 A + a_2 A^{\jury 2} + \dots + a_n A^{\jury n},
\end{equation}
where $A^{\jury \kappa}:=A\jury \cdots \jury A ~\mbox{($\kappa$-times)}$, in which $A^{\jury 0}:=\I_{\jury}$ (defined in~\eqref{E:jury-identity}). Thus, using the properties in \eqref{E:jury-property}, for $p,q\in \C[z]$, we have
\begin{align}\label{E:poly-act-prop}
p_{\jury}(A)\jury q_{\jury}(A) = (pq)_{\jury}(A) \quad \mbox{and}\quad (p+q)_{\jury}(A)=p_{\jury}(A)+q_{\jury}(A),
\end{align}
where $(pq)(z):=p(z)q(z)$ and $(p+q)(z):=p(z)+q(z)$.

In contrast to the Schur product, the entries of $p_{\jury}(A)$ may not be a straightforward computation, and as we later show, it requires an object from combinatorics, called partitions. A partition of a positive integer is a way of writing it as a sum of positive integers. This concept has deep connections to number theory and representation theory. For our requirement, we consider partitions of pairs of integers in $\mathcal{K}_{M,N}^*$.

\begin{defn}[Partitions over $\mathcal{K}_{M,N}^*$]\label{defn:partitions}
For integers $M,N,\ell\geq 1$ and $(i,j)\in \mathcal{K}_{M,N}^*$, define
\[
\mathcal{P}_\ell^*(i,j):=\{\mbox{multiset $S\subseteq \mathcal{K}_{M,N}^*$ of $\ell$ elements $: \sum_{(p,q)\in S}(p,q)=(i,j)$}\}.
\]
Define $c_S(p,q)$ to denote the cardinality of $(p,q)\in S$, for all $\ell,(i,j)$.
\end{defn}

We show later in Theorem~\ref{T:poly-on-mat} exactly how the map
\(p_{\jury}(-): \C^{M \times N} \to \C^{M \times N}\)
is defined in terms of the matrix entries as an extension of a polynomial \(p: \C \to \C\) that meaningfully preserves matrix convolution (Jury product).
With this extension, it is then possible to define a corresponding transform
for any sufficiently differentiable function, as formalized next.

\begin{defn}[Transforms for regular functions]\label{jury-transform-1}
Fix integers $M,N\geq 1$, and let $\F=\C$ or $\R$. For a subset $I\subseteq \F$ and a function $f:I\to \F$ which is $M+N-2$ times differentiable, we define a matrix transform
\[
f_{\jury}(-):I^{M\times N}\to \F^{M\times N}
\]
where for any $A=(a_{ij})\in I^{M\times N}$, the matrix $f_{\jury}(A)\in \F^{M\times N}$ is defined as follows:
\begin{align*}
\big{(}f_{\jury}(A)\big{)}_{00}&:=f(a_{00}), \mbox{ and}\\
\big{(}f_{\jury}(A)\big{)}_{ij}&:=
\sum_{\ell=1}^{i+j} f^{(\ell)}(a_{00})
\sum_{S\in\mathcal{P}_{\ell}^*(i,j)}
\frac{1}{\prod_{(m,n)\in \mathcal{K}_{M,N}^*} c_S(m,n)!}
\prod_{(p,q)\in S} a_{p,q},
\end{align*}
for all $(i,j)\in \mathcal{K}_{M\times N}^*$. Here the product over multiset $S$ is `with' multiplicity.
\end{defn}

Note, $f_\jury(-)$ is a matrix-valued extension of a scalar function $f$, tailored to preserve the structure of convolution (that is clearly unlike the functional and entrywise calculi). This defines our transforms for functions that are sufficiently differentiable, as the construction for polynomials depends on derivatives at a matrix entry (Theorem~\ref{T:poly-on-mat}). We extend this further.

One of our goals is to examine positivity preservers for these matrix transforms. Therefore, as the first entry in $f_{\jury}(A)$ is $f(a_{00})$, a positivity preserver $f_{\jury}(-)$ must be such that $f$ maps nonnegative real numbers to itself. In light of this -- and the wide-spread application of forward difference operators in pure and applied mathematics -- we update Definition~\ref{jury-transform-1}, and define matrix transforms for \textit{any} function, via divided difference operators.

\begin{defn}[Divided differences]\label{defn:div-diff} Under the premise in Definition~\ref{jury-transform-1}, let $I\subseteq \F$ and $f:I\to\F$. Given $h > 0$ and an integer $\ell \geq 0$ such that $x+k h\in I$ for $k\in [0:\ell]$, the $\ell$-th order forward differences with step size $h>0$ are defined as follows:
\begin{align*}
    (\Delta_{h}^{0}f)(x)&:=f(x),\\
    (\Delta_{h}^{\ell}f)(x)&:=(\Delta_{h}^{\ell-1}f)(x+h)-(\Delta_{h}^{\ell-1}f)(x)\\
    &=\sum_{j=0}^{\ell}\binom{\ell}{j}(-1)^{\ell-j}f(x+jh), \quad \mbox{whenever }\ell>0.
\end{align*}
Similarly, the $\ell$-th order divided differences with step size $h>0$ are
\begin{align*}
    (D_{h}^{\ell}f)(x):=\frac{1}{h^{\ell}}(\Delta_{h}^{\ell}f)(x) \quad \mbox{for all } \ell\geq 0.
\end{align*}
\end{defn}

\begin{defn}[Transforms for any function]\label{jury-transform-2}
Under the premise in Definition~\ref{jury-transform-1}, for a subset $I\subseteq \F$ and a function $f:I\to \F$, we define a matrix transform for step size $h>0$,
\[
f_{\jury}(-)_{h}:I^{M\times N}\to \F^{M\times N}
\]
where for a matrix $A=(a_{ij})\in I^{M\times N}$, the matrix $f_{\jury}(A)_{h}\in \F^{M\times N}$ is as follows: provided $a_{00}+kh\in I$ for $k\in [0:M+N-2]$, define
\begin{align*}
\big{(}f_{\jury}(A)_h\big{)}_{00}&:=f(a_{00}), \mbox{ and}\\
\big{(}f_{\jury}(A)_{h}\big{)}_{ij}&:=
\sum_{\ell=1}^{i+j} (D_h^{\ell}f)(a_{00})
\sum_{S\in\mathcal{P}_{\ell}^*(i,j)} \frac{1}{\prod_{(m,n)\in \mathcal{K}_{M,N}^*} c_S(m,n)!}
\prod_{(p,q)\in S} a_{p,q}
\end{align*}
for all $(i,j)\in \mathcal{K}_{M\times N}^*$. Here the product over multiset $S$ is `with' multiplicity.
\end{defn}

\begin{remark}\label{rem:jury-transform-2}

Some immediate remarks are in order.

\begin{enumerate}
    \item One notes that the coefficients inside the inner sum in Definitions \ref{jury-transform-1} and \ref{jury-transform-2} are the multinomial coefficients times $1/\ell!$.
    \item Notice that $f_{\jury}(-)_h$ is naturally a ``weaker'' form of $f_{\jury}(-)$. Therefore, to work with it, ensuring that $f_{\jury}(-)_h$ remains well-defined, we ``strengthened'' it by allowing \textit{all possible} step sizes $h>0$. However, if one already knows that $f \in C^{M+N-2}$, then taking limit yields the (possibly desirable) smooth version (for open $I$; see \cite[Section~3.4]{davis1975interpolation}):
\begin{align}\label{E:taking-lim}
\lim_{h\to 0^+} (D^{\ell}_h f)(a_{00}) = f^{(\ell)}(a_{00}) \implies \lim_{h\to 0^+} f_{\jury}(A)_{h} = f_{\jury}(A).
\end{align}
    \item Finally, our proofs of the positivity preservers (discussed later) will show the ``full extent'' -- i.e., all possible $h>0$ in Definition~\ref{jury-transform-2} -- of $f_{\jury}(-)_{h}$ being used.
\end{enumerate}
\end{remark}

With this summary of the algebraic properties of the Jury product or matrix convolution and its extension as matrix transforms, we are now prepared to discuss its behavior over the cone of positive semidefinite matrices, as well as the positivity preserving properties of the matrix transforms.

\section{Convolution as positivity preservers}\label{S:pos-preservers}

The positive semidefinite matrices in $\C^{N\times N}$ form a convex cone which is further closed under taking entrywise limits. Thus, Theorem~\ref{T:jury} yields that whenever $A \in \C^{N\times N}$ is positive semidefinite,
\[
A^{\jury 0}=\I_{\jury},\quad A^{\jury 1}=A,\quad A^{\jury 2}=A\jury A,\quad  A^{\jury k}=A\jury \cdots \jury A, \quad \sum_{k=0}^{\infty}c_k A^{\jury k}
\]
are all positive semidefinite, for nonnegative real numbers $c_0,c_1,c_2,\dots$, provided there is necessary convergence. This is a P\'olya and Szeg\H{o}-type formulation, for a matrix transform rooted in convolution.

\begin{utheorem}\label{T:polya-szego-for-jury}
Suppose \( f(x) := \sum_{k=0}^{\infty} c_k x^k \), with \( c_k \geq 0 \) for all \( k \), is defined within the domain of convergence $I\subseteq \R$. Then $f_{\jury}(A)$ is positive semidefinite, for all positive semidefinite $A\in I^{N\times N}$, and for all $N\geq 1$.
\end{utheorem}

Therefore, as in the entrywise case, our matrix transforms for a convergent power series with nonnegative Maclaurin coefficients (also called absolutely monotone) preserve positivity. Given this, as did P\'olya and Szeg\H{o} in the entrywise case, we ask a natural question: does there exist a function that is not absolutely monotonic, yet still preserves positivity over matrices of all sizes, when operated according to matrix transforms rooted in convolution? As we show, akin to the results of Schoenberg~\cite{schoenberg1942positive}, Rudin~\cite{Rudin-Duke59}, Vasudeva \cite{vasudeva1979positive}, and others~\cite{khare2022matrix} in entrywise calculus: no such functions exist beyond the absolutely monotonic ones over positive real numbers. This shows our Schoenberg-analogue for convolution viewed as a matrix transform:

\begin{utheorem}\label{T:schoeberg-for-jury}
Let $0< \rho \leq \infty$ and $I=(0,\rho)$. Suppose $\Omega \subseteq \C$ is a subset such that $\Omega\cap (0,\infty) = I$. Consider the following for $f:\Omega \to \C$.
\begin{enumerate}[$(a)$]
    \item The matrix $f_{\jury}(A)_{h}$ is positive semidefinite whenever $A=(a_{ij})\in \Omega^{N\times N}$ is positive semidefinite, for all $h>0$ such that $(a_{00},a_{00}+2(N-1)h)\subseteq I$, and for all $N\geq 1$.
    \item $f:I\to \R$ is $C^{\infty}$, and $f_{\jury}(A):=\lim_{h\to 0^+}f_{\jury}(A)_h$ is positive semidefinite whenever $A\in \Omega^{N\times N}$ is positive semidefinite, for all $N\geq 1$.
    \item We have absolute monotonicity, namely,
    \begin{align*}
        f(x)= \sum_{k=0}^{\infty}c_kx^k \quad \mbox{is convergent,} \quad \mbox{for all } x\in I,
    \end{align*}
    for some nonnegative real numbers $c_0,c_1,c_2,\dots$.
    \item The matrix $f_{\jury}(A)_{h}$ is positive semidefinite whenever $A=(a_{ij})\in \Omega^{N\times N}$ is positive semidefinite, for all sufficiently small $h=h(A)>0$, and for all $N\geq 1$.
\end{enumerate}
Then $(a)\implies (b)$ for $\rho=\infty$, and for any $0<\rho\leq \infty$ we have $(b)\iff (c)\implies (d)$. Moreover, in general $(c)\implies(a)$ may not hold (Example~\ref{Example:T-Schoenberg}).
\end{utheorem}

\begin{remark}[Interpretations]
Theorem~\ref{T:schoeberg-for-jury} addresses two types of matrix transforms: $f_{\jury}(-)_h$ and $f_{\jury}(-)$, and some clarifications are in order.
\begin{enumerate}
    \item If no regularity of $f$ is assumed, we work with $f_{\jury}(-)_h$ as the matrix transform in $(a)$. Then, the positivity-preserving property yields that $f$ is smooth, allowing an ``upgrade'' by taking limit \eqref{E:taking-lim}.
    \item Notice that (contrary to entrywise preservers) $f_{\jury}(-)$ operates via the action of the scalar function and its derivatives on positive reals only. Therefore, in $(b)$ we treat $f_{\jury}(-)$ by considering only the real derivatives, and so we wrote ``$f_{\jury}(A):=\lim_{h\to 0^+}f_{\jury}(A)_h$''. However, for more generality, we considered complex Hermitian matrices $A$.
    \item These transforms preserve positivity precisely when $f$ is absolutely monotone over the positive reals, yielding a Schoenberg-type theorem for convolution. The final assertion treats the case where, even with absolute monotonicity, one still wishes to employ the $h$-dependent transforms $f_{\jury}(-)_h$.
\end{enumerate}
\end{remark}

Notice that Theorem~\ref{T:schoeberg-for-jury} provides a (complete) classification of transforms that preserve positivity across matrices of all dimensions. A natural next step is to investigate how the classification changes when the matrix dimension is fixed. To this end, we now present a sequence of results addressing the (more difficult) fixed-dimensional setting.

\begin{notn}
Given a set $\Omega \subseteq \C$ and integer $N\geq 1$, let $\Pn_N(\Omega)$ denote the set of $N \times N$ positive semidefinite matrices whose entries lie in $\Omega$. We also define $\Pn_N:=\Pn_N(\C)$.
\end{notn}

The implication $(c) \implies (b)$ in Theorem~\ref{T:schoeberg-for-jury} follows from Theorem~\ref{T:polya-szego-for-jury}. The other implications require more work, including our approach that follows the (chronological) development of entrywise positivity preservers. One recalls that the original Schoenberg's Theorem~\ref{Tschoenberg} was proved only for continuous functions, and later Rudin \cite{Rudin-Duke59}, among presenting other generalizations, removed the continuity condition. In parallel to this refinement, we propose an analogue for our transforms defined for any scalar function.

\begin{utheorem}\label{T:Rudin-for-jury}
Suppose $I=(0,\rho)$ is an interval for some $0< \rho \leq \infty$. Let $\Omega \subseteq \C$ such that $\Omega\cap (0,\infty) = I$, and let $f:\Omega\to \C$ be a map. Fix an integer $N\geq 2$, and suppose $f_{\jury}(A)_{h}\in \Pn_N$  for all $A=(a_{ij})\in \Pn_N(\Omega)$, for all $h>0$ such that $(a_{00},a_{00}+2(N-1)h)\subseteq I$. Then $f:I\to \R$ is nonnegative, nondecreasing, and continuous. Moreover, when $N\geq 3$ and $\rho=\infty$, we have that $f:I\to \R$ is convex (Definition~\ref{defn:mid-con}).
\end{utheorem}

Another pivotal contribution to entrywise positivity preservers is due to Horn \cite{horn1969theory}: building on an idea of Loewner, he proved that if $f: (0,\infty) \to \R$ is continuous and $f[A]:=(f(a_{ij}))\in \Pn_N$ for all $A:=(a_{ij})\in \Pn_N((0,\infty))$, then $f\in C^{N-3}((0,\infty))$ and $f^{(k)}(x) \geq 0$ for all $x > 0$ and all $0 \leq k \leq N-3$. Moreover, if one already knows that $f\in C^{N-1}((0,\infty))$, then $f^{(k)}(x)\geq 0$ for all $x>0$ and all $0\leq k\leq N-1$. (See \cite{guillot2017preserving} for refinements.) Here, we present the analogue of this for both of our matrix transforms.

\begin{utheorem}\label{T:horn-for-jury}
Fix an integer $N\geq 2$, and let $f:(0,\infty)\to \R$.
\begin{enumerate}[$(a)$]
    \item Suppose $f_{\jury}(A)_h\in \Pn_N$ for all $A\in \Pn_N((0,\infty))$ and all $h>0$. Then $f\in C^{N-3}((0,\infty))$, and
    \[
    f^{(k)}(x)\geq 0 \quad \mbox{for all}\quad x\in (0,\infty) \quad \mbox{and all}\quad 0\leq k\leq N-3.
    \]
    Moreover, $f^{(k)}$ for $0\leq k\leq N-3$ are convex, and $f^{(N-3)}$ has nondecreasing left- and right-hand derivatives.
    \item Suppose $f$ is $2N-2$ times differentiable over $(0,\infty)$, and $f_{\jury}(A)\in \Pn_N$ for all $A\in \Pn_N((0,\infty))$. Then,
    \[
    f^{(k)}(x)\geq 0 \quad \mbox{for all}\quad x\in (0,\infty) \quad \mbox{and all}\quad 0\leq k\leq N-1.
    \]
    Moreover, this holds if $(0,\infty)$ is replaced by $(0,\rho)$ for any $0<\rho\leq \infty$.
\end{enumerate}
\end{utheorem}

Theorem~\ref{T:horn-for-jury} involves two matrix transforms, $f_{\jury}(-)_h$ and $f_{\jury}(-)$, whose applicability depends on the regularity of $f$. When no regularity is assumed, we employ $f_{\jury}(-)_h$, as in Theorem~\ref{T:horn-for-jury}$(a)$. By contrast, when $f$ is sufficiently differentiable, the transform $f_{\jury}(-)$ can be applied directly. Both yielding results in analogy with the classical Horn theorem. We note, however, that our assumption is stronger: ${2N-2}$ times differentiability rather than $C^{N-1}$ in Horn’s theorem. This is to ensure that $f_{\jury}(-)$ is well defined on $N \times N$ matrices.

Following the classical Schoenberg Theorem~\ref{Tschoenberg}, several works have focused on understanding entrywise transforms that preserve the positivity of matrices of a fixed dimension \(N \ge 2\), rather than for all dimensions simultaneously. A seminal result in this direction is due to FitzGerald and Horn \cite{fitzgerald1977fractional} on fractional Hadamard powers. They showed that for any fixed \(N \geq 2\) and \(\alpha \in \mathbb{R}\), the matrix \(A^{\circ \alpha} := (a_{ij}^\alpha)\) is positive semidefinite whenever \(A = (a_{ij})\) is so, with \(a_{ij} \ge 0\), if and only if $\alpha \in \mathbb{Z}_{\ge 0}$ or \(\alpha \geq N-2\). This work introduced several useful techniques including the well-known ``integration trick'', inspired recent works connecting analysis and combinatorics \cite{guillot2016critical}, and highlighted a key distinction from Schoenberg’s Theorem~\ref{Tschoenberg}: in the fixed-dimensional setting, there \textit{do} exist functions that are $C^{\infty}$ and not absolutely monotonic yet still preserve positivity. In a parallel development towards this for our matrix transforms, we present the following result.

\begin{utheorem}\label{T:fh-for-jury}
Suppose $I=(0,\rho)$ for $0<\rho\leq \infty$, and let $N\geq 2$ be a fixed integer. For $\alpha\in \R$ define $f(x)\equiv x^{\alpha}$.
\begin{enumerate}[$(a)$]
    \item For $N=2$: the matrix $f_{\jury}(A)\in \Pn_2$ for all $A\in \Pn_2(I)$ if and only if $\alpha\geq 0$.
    \item For $N\geq 3$: for each $\alpha<N-2$ that is not a nonnegative integer, there exists $A\in \Pn_N(I)$ such that $f_{\jury}(A)\not\in \Pn_N$.
\end{enumerate}
\end{utheorem}

\begin{remark}
Theorem~\ref{T:fh-for-jury} provides a partial analogue of the FitzGerald and Horn result for matrix convolution, and a few observations are in order. First, contrary to Theorem~\ref{T:schoeberg-for-jury}, in the fixed-dimensional case of $N=2$, there exist positivity preservers that are $C^{\infty}$ and not absolutely monotonic. For higher dimensions $N \geq 3$, we have obtained partial results, again in parallel with the classical theorem of FitzGerald--Horn. (For a possible way to resolve the remaining open cases, see Section~\ref{S:concluding} for a more qualitative meaning of our transforms.)
\end{remark}

This summarizes our results on positivity preservers for transforms that respect convolution or the Jury product. We conclude with a note that the proofs of these aforementioned results seem to be using the ``full extent'' of $f_{\jury}(-)_h$, and rely on several fundamental tools from real and complex analysis. These include Bernstein’s theorem \cite{bernstein1929sur} on absolutely monotonic functions, the implications of Boas and Widder's theorem \cite{boas1940functions} for continuous functions with nonnegative divided differences, as well as the roles of monotonicity and nonnegativity in establishing continuity of a function.

Finally, as previously noted, the definition of our transforms is based on a Cayley--Hamilton-type theorem for matrix spaces equipped with Jury's product / convolution. We systematically develop this in the next section.

\section{Cayley--Hamilton theorem}\label{S:cayley-hamilton-thm}

The classical Cayley--Hamilton theorem states that the characteristic polynomial of a square matrix annihilates the matrix \cite{dummit2004abstract}. More explicitly, for an integer $N \geq 1$ and $A \in \mathbb{C}^{N \times N}$, we have $p_A(A) = 0$ if we know that
\begin{align}\label{eq:ch-std}
p_A(z) = \det (zI - A) \quad \mbox{where} \quad \deg(p_A) = N.
\end{align}
While a specific matrix may satisfy a polynomial identity of degree less than $N$, the upper bound $N$ is attainable, as in the case of diagonal matrices with distinct diagonal entries.

This observation naturally raises the question of how the Cayley--Hamilton framework changes when matrices are equipped with products other than the standard one. For example, in the case of the Schur (entrywise) product of (possibly rectangular) matrices, if $A=(a_{ij})\in \C^{M\times N}$, the corresponding Cayley--Hamilton theorem takes the form
\begin{align}\label{eq:ch-schur}
p_{A}(z) = \prod_{i=0}^{M-1}\prod_{j=0}^{N-1} (z-a_{ij}) \quad \mbox{where} \quad \deg(p_A) = MN,
\end{align}
and the maximal degree $MN$ is realized when the entries of $A$ are all distinct. This already illustrates that the structure of the underlying product has a decisive influence on the form of the Cayley--Hamilton theorem. Moreover, it suggests that the theory can be meaningfully extended to other algebraic settings.

In this spirit, we now turn to the {Jury product} (or convolution of matrices), which applies naturally to rectangular matrices, and establish the corresponding Cayley--Hamilton theorem. To be more specific about our anticipation of it for the rings equipped with the Jury product, we formulate the following question.

\begin{quest}\label{question}
Fix integers $M,N \geq 1$.
\begin{enumerate}
    \item Given a matrix $A\in \C^{M\times N}$, does there exist a polynomial $p\in \C[z]$ such that $p_{\jury}(A)=0$?
    \item As $\C^{M\times N}$ is finite dimensional, the answer to above is naturally ``yes''. So we upgrade: does there exist $\kappa \geq 1$ such that for every $A \in \C^{M\times N}$ one can find a monic $p=p_A \in \C[z]$ of degree $\kappa$ with $p_{\jury}(A)=0$?
    \item If yes, what is the smallest value of such a $\kappa$? And, is there an explicit formula for $p_{A}$?
\end{enumerate}
We call the minimal $\kappa$ as the Cayley–Hamilton \textit{degree}, and the corresponding $p_A\in\C[z]$ as the Cayley–Hamilton \textit{polynomial} (for the matrix ring equipped with convolution and the usual sum).
\end{quest}

The following result provides the answers.

\begin{utheorem}\label{main-thm-1}
Fix integers $M,N\geq 1$, and consider the ring of matrices $\C^{M\times N}$ equipped with convolution (the Jury product) and the usual sum. Then the Cayley--Hamilton degree is $M+N-1$. Moreover, for any $A=(a_{ij})\in \C^{M\times N}$, we have
\[
p_{A}(z):=(z-a_{00})^{M+N-1}
\]
as the Cayley--Hamilton polynomial.
\end{utheorem}

\begin{remark}\label{rem:ch-thm-1}
We record a few (immediate) observations.
\begin{enumerate}[$(i)$]

\item In the spirit of the classical Cayley--Hamilton theorem: the polynomial
\(
p_{A}(z) := (z-a_{00})^{M+N-1}
\)
annihilates $A = (a_{ij}) \in \mathbb{C}^{M\times N}$ in the commutative ring $\mathbb{C}^{M\times N}$ endowed with the matrix convolution. In other words:
\[
0=\big{(}A-a_{00}\I_{\jury}\big{)}^{\jury(M+N-1)} = \sum_{j=0}^{M+N-1} \binom{M+N-1}{j} (-a_{00})^{M+N-1-j} A^{\jury j}.
\]
Therefore, we get a formula for the inverse under convolution if $a_{00}\neq 0$,
\begin{align}
A^{\jury -1}= - \sum_{j=1}^{M+N-1} \binom{M+N-1}{j} (-1/a_{00})^{j} A^{\jury (j-1)}.
\end{align}
Compared to the algorithmic \eqref{E:jury-inverse}, the above is direct.

\item Moreover, as will become clear from the proof of Theorem~\ref{main-thm-1}, it remain valid when $\mathbb{C}$ is replaced by any unital commutative ring of characteristic zero.

\item The polynomial $p_A$ need not be minimal among those satisfying $p(A)=0$. The question of the minimal polynomial is subtler, as it depends intricately on the specific matrix $A$, and will be addressed later.

\item Finally, as discussed later in Subsection~\ref{R:CH-for-unbound-mat}, a striking contrast to the standard and Schur products arises: the answer to Question~\ref{question}(1) is ``no'' for the natural extension $\C_{00}^{\N\times \N}$ of $\C^{M\times N}$. However, on the Hilbert space $\overline{\C_{00}^{\N\times \N}}$, the answer is ``no'' for all the three matrix products.
\end{enumerate}
\end{remark}

\subsection*{Organization of this section}
As the Jury product (matrix convolution) has rarely been studied from an algebraic perspective, we adopt a more gradual approach. The remainder of this section is therefore developed in a step-by-step manner, examining successive facets of the Jury product one at a time: from ``visual'' or ``combinatorial'' to more ``rigorous''. This progression may be observed as the discussion moves from Lemma~\ref{Cayley-Hamilton-Jury-2by2-lemma-0} ${\bf \longrightarrow}$ Lemma~\ref{lemma-AB-prod-1}
${\bf \longrightarrow}$ Lemma~\ref{k-power-entries:thm} ${\bf \longrightarrow}$ Theorem~\ref{T:poly-on-mat}.

We begin with the proof of Theorem~\ref{main-thm-1} in the special case $M = N = 2$. Although this case is highly specific, the proof is instructive: it both highlights certain intrinsic features of the Jury product (that derivatives appear naturally) and provides essential background for the general argument, which is more intricate and is presented in the later part of Subsection~\ref{Ssub:cayley-hamilton-proof}. The discussion then turns to the minimal polynomial in Subsection~\ref{Ssub:min-poly}. Together, these show that the new matrix transforms in Definitions~\ref{jury-transform-1} and \ref{jury-transform-2} arise as natural extensions of scalar functions to matrix spaces equipped with convolution.

\subsection{Proof of the Cayley--Hamilton-type theorem}\label{Ssub:cayley-hamilton-proof}

We first consider $M=N=2$.

\begin{lemma}\label{Cayley-Hamilton-Jury-2by2-lemma-0}
Let $A:=\begin{pmatrix}a & b \\ c & d\end{pmatrix} \in \C^{2\times 2}$. Then the $k$-th Jury power of $A$ is given by
\begin{align*}
A^{\jury k}=\begin{pmatrix}a^k & ka^{k-1}b \\ ka^{k-1}c & ka^{k-1}d + k(k-1)a^{k-2}bc\end{pmatrix}
\end{align*}
for all integers $k\geq 0$.
\end{lemma}
\begin{proof}
It is obvious that the formula holds for $k=0,1$. We employ induction for $k+1\geq 2$. We compute using the associativity of the Jury product:
\begin{align*}
    A^{\jury (k+1)}=A^{\jury k}\jury A &= \begin{pmatrix}a^k & ka^{k-1}b \\ ka^{k-1}c & ka^{k-1}d + k(k-1)a^{k-2}bc\end{pmatrix} \jury \begin{pmatrix}a & b \\ c & d\end{pmatrix}\\
    &= \begin{pmatrix}a^{(k+1)} & (k+1)a^{k}b \\ (k+1)a^{k}c & (k+1)a^{k}d + (k+1)ka^{k-1}bc\end{pmatrix}.
\end{align*}
This completes the proof.
\end{proof}

\begin{proof}[Proof of Theorem~\ref{main-thm-1} for $M=N=2$]
From Lemma~\ref{Cayley-Hamilton-Jury-2by2-lemma-0} and \eqref{defn:p-of-A-for-jury}, for all $p\in \C[z]$,
\begin{align}\label{Cayley-Hamilton-Jury-2by2-lemma-remark}
    p_{\jury}(A)=\begin{pmatrix}
        p(a) & b p'(a) \\
        c p'(a) & dp'(a) + bc p''(a)
    \end{pmatrix} \qquad \mbox{for all } A\in \C^{2\times 2}.
\end{align}
Therefore if $p(z):=(z-a)^{3}$ then $p_{\jury}(A)=0$. Thus the required Cayley--Hamilton degree can not exceed $3$, and the formula for the required polynomial is valid. Finally, if $b,c\neq 0$ then $q_{\jury}(A)\neq 0$ for $q$ either $(z-a)$ or $(z-a)^{2}$. That there can be no other linear or quadratic polynomials annihilating $A$ is obvious, completing the proof.
\end{proof}

\subsection*{Exploring a Hankel-type structure} The resolution of $M=N=2$ case shows that the action of a polynomial in $\C[z]$ over $\C^{2\times 2}$ involves derivatives at the first entry $a_{00}$ of the matrices. Moreover, there is a ``symmetric'' nature of Jury's product which plays an important role in extending the above case for the remaining cases. To formalize this, we need the following notation for integers $M,N\geq 1$ (which extends Definition~\ref{defn:jury-prod}(a)).

For $\ell=1,\dots, M+N-1$, define:
\begin{align}
\mathcal{K}^{(\ell)}_{M,N}:=\big{\{}(i,j)\in \mathcal{K}_{M,N}:i+j\leq \ell-1\big{\}}
\end{align}

Using this, we look at a matrix $A:=(a_{ij})\in \C^{M\times N}$ by observing its \textit{anti-diagonals}:
\begin{center}
\begin{tabular}{@{} >{\centering\arraybackslash}m{0.40\textwidth} @{\quad} >{\centering\arraybackslash}m{0.60\textwidth} @{}}
$\displaystyle
\begin{pmatrix}
a_{00} & a_{01} & a_{02} & \iddots \\
a_{10} & a_{11} & \iddots & \iddots \\
a_{20} & \iddots & \iddots & \iddots \\
\iddots & \iddots & \iddots &  \ddots
\end{pmatrix}$
&
\begin{minipage}{\linewidth}
the first anti-diagonal contains $a_{00}$,\\[2pt]
then the second contains $a_{10},a_{01}$,\\[2pt]
the third contains $a_{20},a_{11},a_{02}$,\\[2pt]
\(\vdots\)\\[2pt]
the \(\ell\)-th one contains \(a_{ij}\) for \(i+j=\ell-1\).
\end{minipage}
\end{tabular}
\end{center}

So there are exactly $M+N-1$ anti-diagonals in $A\in \C^{M\times N}$, and the indices in $\mathcal{K}^{(\ell)}_{M,N}$ refer to all the entries $a_{ij}$ up to the $\ell$-th anti-diagonal of $A$. This construction yields basic results for the Jury product in which the role of these anti-diagonals is crucial (which eventually helps in showing that the required Cayley--Hamilton degree can not exceed $M+N-1$).

\begin{lemma}\label{lemma-AB-prod-1}
Fix integers $M,N\geq 1$ and let $A:=(a_{ij})\in \C^{M\times N}$.
\begin{enumerate}[$(1)$]
\item For $\ell\in [1:M+N-2]$, suppose $B:=(b_{ij})\in \C^{M\times N}$ with $b_{ij}=0$ for all $(i,j)\in \mathcal{K}^{(\ell)}_{M,N}$. If $a_{00}=0$, then $(A\jury B)_{ij}=0$ for all $(i,j)\in \mathcal{K}^{(\ell+1)}_{M,N}$.
\item If $(z-a_{00})^{M+N-1}$ divides $p(z)\in \C[z]$, then $p_{\jury}(A)=0$.
\end{enumerate}
\end{lemma}
\begin{proof} We begin with the proof of the first assertion.
\begin{enumerate}[$(1)$]
\item For every $(i,j)\in \mathcal{K}^{(\ell+1)}_{M,N}$, the entries $(A\jury B)_{ij}$ are sums over products of entries of $B$ indexed by $\mathcal{K}^{(\ell)}_{M,N}$ with entries of $A$, and the entry $a_{00}=0$ of $A$ with the entry $b_{ij}$ with $i+j= \ell+1$. Each of these products are therefore zero, and thus $(A\jury B)_{ij}=0$.

\item Considering \eqref{defn:p-of-A-for-jury}, it is sufficient to show that $A^{\jury (M+N-1)} = 0$ if $a_{00}=0$. It is obvious for $M=N=1$, and Lemma~\ref{Cayley-Hamilton-Jury-2by2-lemma-0} implies it for $M=N=2$. We use the following argument for other values of $M,N$.

As we mentioned earlier, we look at $A=(a_{ij})$ by observing its anti-diagonals. Assuming $a_{00}=0$, we claim that for $\ell \geq 1$, the entries on the $k$-th anti-diagonal for $k=1,2,\dots,\ell$ of $A^{\jury \ell}$ are all zeros. This is easy to verify for $\ell = 1,2$ (which we skip). We use induction to prove it for $\ell\geq 3$. Note that $A^{\jury \ell} = A^{\jury (\ell-1)}\jury A$. By induction, all the entries up to $(\ell-1)$-th anti-diagonal of $A^{\jury (\ell -1)}$ are zeros. Since $a_{00}=0$ and $A^{\jury \ell} = A^{\jury (\ell-1)}\jury A$, Lemma~\ref{lemma-AB-prod-1}$(1)$ implies that the $\ell$-th anti-diagonals of $A^{\jury \ell}$ are also zeros.

Finally, observe that all the entries in $A^{\jury (M+N-1)}$ (where $a_{00}=0$) are within the $M+N-1$ anti-diagonals in $A^{\jury (M+N-1)}$. Thus they all vanish, completing the proof.\qedhere
\end{enumerate}
\end{proof}

Lemma~\ref{lemma-AB-prod-1} therefore shows that the Cayley--Hamilton degree can not exceed $M+N-1$, and that polynomials $p_{A}(z):=(z-a_{00})^{M+N-1}$ are good candidates for the Cayley--Hamilton polynomials. Now to show that the degree is exactly $M+N-1$, we need some `combinatorial' preparation involving the formula for the $\kappa$-fold Jury product of matrices. We need the following notion of partitions over $\mathcal{K}_{M,N}$ (which is the nonnegative version of the one in Definition~\ref{defn:partitions}).

\begin{defn}[Partitions over $\mathcal{K}_{M,N}$]\label{k-power-entries:defn-1}
For integers $M,N,\kappa\geq 1$, and $(i,j)\in \mathcal{K}_{M,N}$, define
\[
\mathcal{P}_\kappa(i,j):=\{\mbox{multiset $S\subseteq \mathcal{K}_{M,N}$ of $\kappa$ elements $: \sum_{(p,q)\in S}(p,q)=(i,j)$}\}.
\]
Define $c_S(p,q)$ to denote the cardinality of $(p,q)\in S$, for all $\kappa,(i,j)$.
\end{defn}

We are ready to state the formula for the $\kappa$-th Jury power of any matrix in $\C^{M\times N}$. It generalizes Lemma~\ref{Cayley-Hamilton-Jury-2by2-lemma-0}, and involves the well-known multinomial coefficients.

\begin{lemma}\label{k-power-entries:thm}
Suppose $M,N,\kappa\geq 1$ are integers.
\begin{enumerate}[$(1)$]
    \item Fix $A_{p}:=(a_{ij}^{(p)})\in \C^{M\times N}$ for $p\in [1:\kappa]$. Then $A_1 \jury \dots \jury A_{\kappa}$ is given by:
\begin{align*}
    (A_1 \jury \dots \jury A_{\kappa})_{ij}=\sum_{\substack{s_1(i) + \dots + s_\kappa(i) = i \\ s_1(j) + \dots + s_\kappa(j) = j}}
\prod_{p=1}^\kappa a_{s_{p}(i),s_{p}(j)}^{(p)},
\end{align*}
for all $(i,j)\in \mathcal{K}_{M,N}$, where all $s_p(i),s_p(j)\geq 0$ are integers.
\item For $A=(a_{ij})\in \C^{M\times N}$, the formula for $A^{\jury \kappa}$ is given by:
\begin{align*}
\begin{pmatrix}A^{\jury \kappa}\end{pmatrix}_{ij}=\sum_{S\in \mathcal{P}_{\kappa}(i,j)} \frac{\kappa!}{\prod_{(m,n)\in \mathcal{K}_{M,N}}{c}_S(m,n)!} \prod_{(p,q)\in S} a_{p,q},
\end{align*}
for all $(i,j)\in \mathcal{K}_{M,N}$, and where the product over multiset $S$ is `with' multiplicity.
\end{enumerate}
\end{lemma}
\begin{proof} We begin with the proof of the first assertion.
\begin{enumerate}[$(1)$]
    \item For this proof, possibly a convenient way to state the result is the following: the Jury product $A_1 \jury \dots \jury A_{\kappa}\in \C^{M\times N}$ is given by
\begin{align}\label{k-power-entries:lemma-1:eqn-1}
    (A_1 \jury \dots \jury A_{\kappa})_{ij}=\sum_{S\in \mathcal{P}_\kappa'(i,j)} \prod_{p=1}^{\kappa} a_{s_p(i),s_p(j)}^{(p)}
\end{align}
where $S:=\big{(}(s_1(i),s_1(j)),\dots,(s_\kappa(i),s_\kappa(j))\big{)}\in \mathcal{P}_{\kappa}'(i,j)\subseteq \mathcal{K}_{M,N}^{k}$ is a $\kappa$ element ordered multi-subset such that $\sum_{(p,q)\in S}(p,q)=(i,j)$. We prove this equivalent statement.

The claim is true for $\kappa=1,2$. Suppose it holds for all $\ell \in [1:\kappa-1]$; we show that it holds for $\kappa\geq 3$. In each $\prod_{p=1}^{\kappa} a_{s_p(i),s_p(j)}^{(p)}$ on the right hand side of \eqref{k-power-entries:lemma-1:eqn-1}, the entry $a^{(\kappa)}_{s_\kappa(i),s_\kappa(j)}$ is such that $({s_\kappa(i),s_\kappa(j)})\in [0:i]\times [0:j]$. Therefore, the right hand side of \eqref{k-power-entries:lemma-1:eqn-1} can be re-written as
\begin{align*}
    \sum_{S\in \mathcal{P}_\kappa'(i,j)} \prod_{p=1}^{\kappa} a_{s_p(i),s_p(j)}^{(p)} = \sum_{\alpha=0}^{i}\sum_{\beta=0}^{j} \bigg{(} \sum_{S\in \mathcal{P}_{\kappa-1}'(\alpha,\beta)} \prod_{p=1}^{\kappa-1} a_{s_p(i),s_p(j)}^{(p)}\bigg{)} a^{(\kappa)}_{i-\alpha,j-\beta}
\end{align*}
where $\mathcal{P}_{\kappa-1}'(\alpha,\beta)$ is the collection of all $\kappa-1$ element ordered multi-subset $S$ of elements from $\mathcal{K}_{M,N}$ such that $\sum_{(p,q)\in S}(p,q)=(\alpha,\beta)$, where $S:=\big{(}(s_1(i),s_1(j)),\dots,(s_{\kappa-1}(i),s_{\kappa-1}(j))\big{)}$. Therefore we have,
\[
(A_1 \jury \dots \jury A_{\kappa-1})_{\alpha,\beta}=\sum_{S\in \mathcal{P}_{\kappa-1}'(\alpha,\beta)} \prod_{p=1}^{\kappa-1} a_{s_p(i),s_p(j)}^{(p)}
\]
And so, we have
\begin{align*}
&\sum_{\alpha=0}^{i}\sum_{\beta=0}^{j} \bigg{(} \sum_{s\in \mathcal{P}_{\kappa-1}'(\alpha,\beta)} \prod_{p=1}^{\kappa-1} a_{s_p(i),s_p(j)}^{(p)}\bigg{)} a^{(\kappa)}_{i-\alpha,j-\beta}\\
= &\sum_{\alpha=0}^{i}\sum_{\beta=0}^{j} (A_1 \jury \dots \jury A_{\kappa-1})_{\alpha,\beta}~ a^{(\kappa)}_{i-\alpha,j-\beta}.
\end{align*}
Then we use induction for $\kappa'=2$, and the associativity of Jury's product to obtain that above equals:
\begin{align*}
\big{(}(A_1 \jury \dots \jury A_{\kappa-1})\jury A_{\kappa}\big{)}_{ij}&= (A_1 \jury \dots \jury A_{\kappa-1}\jury A_{\kappa})_{ij}.
\end{align*}

\item Take $A=A_1=\dots=A_\kappa$ and note that the term $\prod_{p=1}^{\kappa} a_{s_p(i),s_p(j)}^{(p)}$ for each ordered multiset $S\in \mathcal{P}_{\kappa}'(i,j)$ repeats exactly \[
\frac{\kappa!}{\prod_{(m,n)\in \mathcal{K}_{M,N}}{c}_S(m,n)!}\]
many times, completing the proof.\qedhere
\end{enumerate}
\end{proof}

As a preliminary exercise and a sanity check (before we prove the main result) we first present a proof of Lemma~\ref{lemma-AB-prod-1}(2) via Lemma~\ref{k-power-entries:thm}.

\begin{proof}[Proof of Lemma~\ref{lemma-AB-prod-1}(2) via Lemma~\ref{k-power-entries:thm}] We aim to show that if $A:=(a_{ij})\in \C^{M\times N}$ with $a_{00}=0$, then $A^{\jury (M+N-1)}=0$ (which is sufficient). The result follows once we see that for each $S\in \mathcal{P}_{M+N-1}(i,j)$, we must have $(0,0)\in S$. This is so, since $\kappa=M+N-1$ is large enough, so that one has to include $(0,0)\in S$ for $S$ to have $M+N-1$ elements (counted with multiplicities), for each $(i,j)\in \mathcal{K}_{M,N}$. As an example, consider the worst case scenario when $(i,j)=(M-1,N-1)$. Then to avoid $(0,0)\in S$, $S$ must contain $M-1$ copies of $(1,0)$, and $N-1$ copies of $(0,1)$. However these additions only make $M+N-2$ elements in $S$; so it requires one more element. Therefore $S$ must also contain $(0,0)$, concluding the proof.
\end{proof}

These preparations (finally$!$) have brought us to the following main proof:

\begin{proof}[Proof of Theorem~\ref{main-thm-1}]
We know from Lemma~\ref{lemma-AB-prod-1}(2) that $p_{A}(z)=(z-a_{00})^{M+N-1}$ annihilates $A=(a_{ij})$. Therefore the Cayley--Hamilton degree $\leq M+N-1$. That it is exactly $M+N-1$ follows using a matrix with all entries $1$, except zero at the first position, i.e., suppose
\[
A:=\I_{\circ}-\I_{\jury} \in \C^{M\times N}
\]
where $\I_{\circ}$ and $\I_{\jury}$ respectively are the identity for the Schur and the Jury products. Then the Jury powers $A^{\jury \ell}\neq 0$ for $\ell=1,\dots,M+N-2$. For completeness, we give a proof using (the notations in) Lemma~\ref{k-power-entries:thm}(2). For any $(i_0,j_0)\in \mathcal{K}_{M,N}$ with $i_0+j_0=\ell$, we have that there exists $S\in \mathcal{P}_{\ell}(i_0,j_0)$ such that $(0,0)\not\in S$. Namely, this $S$ is given by the collection $\Big{\{} \{(1,0)^{(i_0)},(0,1)^{(j_0)}\} \Big{\}}$. Therefore, since the entries of the aforementioned $A$ are positive for all $(i,j)\neq (0,0)$, we have that $(A^{\jury \ell})_{i_0,j_0}\neq 0$ for all $(i_0,j_0)\in \mathcal{K}_{M,N}$ with $i_0+j_0=\ell$. Therefore the Cayley--Hamilton degree is exactly $M+N-1$.
\end{proof}

\begin{remark}[Extension to unital commutative rings]\label{Rem:ext-to-rings}
We conclude this subsection with a remark regarding the proof of Theorem~\ref{main-thm-1} over any unital commutative ring $R$ of characteristic zero (in place of $\C$). Note that every step for the proof of Theorem~\ref{main-thm-1} goes through verbatim for $R$, except possibly the semi-final sentence in the proof of Theorem~\ref{main-thm-1} immediately above. For $R$, this argument changes to: ``$\dots$ since the entries of the aforementioned $A$ are all $1$ for all $(i,j)\neq (0,0)$, and that the ring has characteristic zero, we have that $(A^{\jury \ell})_{i_0,j_0}\neq 0$ for all $(i_0,j_0)\in \mathcal{K}_{M,N}$ with $i_0+j_0=\ell$.''
\end{remark}

\subsection{Minimal polynomials} \label{Ssub:min-poly}

Notice that the preceding discussion of Cayley--Hamilton polynomial and degree relied only on essentially extending the Jury product (or convolution of matrices) to its $\kappa$-fold product formula (beginning with the ``visual'' Lemma~\ref{lemma-AB-prod-1} then to Lemma~\ref{k-power-entries:thm} which is more explicit). For the more delicate minimal polynomials, however, it is necessary to understand these through more quantitative identities relating polynomials and matrices (Theorem~\ref{T:poly-on-mat}). We begin with a formal definition.

\begin{defn}[Minimal Cayley--Hamilton polynomials for the Jury product]
Fix integers $M,N\geq 1$, and let $A\in \C^{M\times N}$. We call a nonzero monic $p\in \C[z]$ a minimal Cayley--Hamilton polynomial of $A$ (for the Jury product) if $p_{\jury}(A)=0$, and there does not exist a monic nonzero $q\in \C[z]$ with $\deg q<\deg p$ such that $q_{\jury}(A)=0$.
\end{defn}

The algebraic structure induced by the Jury product over $\C^{M \times N}$ shows that a minimal polynomial must be of the form $(z-a_{00})^{\kappa}$ for any given $A:=(a_{ij})\in \C^{M\times N}$ for some integer $\kappa\geq 1$. Formally:

\begin{lemma}\label{Cayley-Hamilton-Jury-kbyk-thm}
Suppose $M,N\geq 1$ are integers. If $p(z)\in \C[z]$ is a minimal Cayley--Hamilton polynomial for $A:=(a_{ij})\in \C^{M\times N}$ for the Jury product, then $p(z)=(z-a_{00})^{\kappa}$ for some $\kappa\in [1:M+N-1]$.
\end{lemma}
\begin{proof}
We know that for $t(z):=(z-a_{00})^{M+N-1}$, we have $t_{\jury}(A)=0$. Therefore $\deg p \leq \deg t=M+N-1$. Then there exist unique $s,r\in \C[z]$ with either $r=0$ or $\deg r<\deg p$, such that $t(z)=p(z)s(z)+r(z)$. If $r\neq 0$ then we have
\begin{align*}
    t_{\jury}(A)=p_{\jury}(A)\jury s_{\jury}(A)+r_{\jury}(A) \implies r_{\jury}(A)=0.
\end{align*}
Since $p\neq 0$ and of the smallest degree with $p_{\jury}(A)=0$, this is a contradiction. Therefore $r=0$. Moreover, as $t$ has a unique linear factor, $p$ has the required form.
\end{proof}

The computation of the degree $\kappa$ of the minimal polynomial may have a delicate relation with matrix entries. For instances, see the following remark.

\begin{remark}\label{Cayley-Hamilton-Jury-2by2-lemma-remark-2}
Recalling \eqref{Cayley-Hamilton-Jury-2by2-lemma-remark} and the proofs of Lemma~\ref{Cayley-Hamilton-Jury-2by2-lemma-0}, it can be seen that the minimal polynomial $p\in \C[z]$ for $A:=\begin{pmatrix}a & b \\ c & d\end{pmatrix}\in \C^{2\times 2}$ is
\begin{align*}
    p(z):=\begin{cases}
        z-a & \mbox{ if }b=c=d=0,\\
        (z-a)^2 & \mbox{ if either exactly one of $b,c=0$, or $b=c=0\neq d$,}\\
        (z-a)^3 &\mbox{ if none of $b,c$ is zero}.
    \end{cases}
\end{align*}
\end{remark}

The minimal polynomials for general $M, N \geq 1$ requires a quantitative understanding on how a generic polynomial acts on a matrix in the ring equipped with the Jury product, thereby yielding the necessary polynomial-matrix identities. Both aspects are formalized in the next result.

\begin{utheorem}\label{T:poly-on-mat}
Fix integers $M,N\geq 1$, and $A=(a_{ij})\in \C^{M\times N}$.
\begin{enumerate}[$(a)$]
    \item For a polynomial $f\in \C[z]$, we have $\big{(}f_\jury(A)\big{)}_{00}=f(a_{00})$, and for all nonzero $(i,j)\in \mathcal{K}_{M,N}$,
\begin{align*}
   \big{(}f_{\jury}(A)\big{)}_{ij} = \sum_{\ell = 1}^{i+j} f^{(\ell)}(a_{00}) \sum_{S\in \mathcal{P}_{\ell}^*(i,j)} \frac{1}{\prod_{(m,n)\in \mathcal{K}_{M,N}^*} c_S(m,n)!} \prod_{(p,q)\in S} a_{p,q}
\end{align*}
where the product over multiset $S$ is `with' multiplicity.
\item The polynomial $f(z):=(z-a_{00})^{\kappa}$ is the minimal polynomial for $A$ for the Jury product, if $\kappa\in [1:M+N-2]$ is the smallest such that
\begin{align*}
    \sum_{S\in \mathcal{P}_{\kappa}^*(i,j)} \frac{1}{\prod_{(m,n)\in \mathcal{K}_{M,N}^*} c_S(m,n)!} \prod_{(p,q)\in S} a_{p,q} = 0
\end{align*}
for all $(i,j)\in \mathcal{K}_{M,N}$ with $i+j\geq\kappa$, where the product over multiset $S$ is `with' multiplicity. Moreover, if no such $\kappa \in [1:M+N-2]$ exists, then $(z-a_{00})^{M+N-1}$ is the minimal polynomial.
\end{enumerate}
\end{utheorem}

\begin{proof} We begin with the proof of the first assertion.
\begin{enumerate}[$(a)$]
\item Suppose $f(z)=z^\eta$ for integer $\eta\geq 1$. Then $\big{(}f_{\jury}(A)\big{)}_{00}=a_{00}^{\eta}=f(a_{00})$. Now for $i+j\geq 1$, compute the following using Lemma~\ref{k-power-entries:thm} and partitioning the sum over different values of $c_{S}(0,0)$, which is the cardinality of $(0,0)$ in $S$:
\begin{align*}
f_{\jury}(A)_{ij}&=(A^{\jury \eta})_{ij} = \sum_{S\in\mathcal{P}_{\eta}(i,j)}
\frac{\eta!}{\prod_{(m,n)\in \mathcal{K}_{M,N}} c_S(m,n)!}
\prod_{(p,q)\in S} a_{p,q} \\
& = \sum_{\ell = 1}^{\min(\eta,i+j)} \sum_{\substack{S\in \mathcal{P}_{\eta}(i,j):\\c_S(0,0)=\eta-\ell}} \frac{\eta!}{\prod_{(m,n)\in \mathcal{K}_{M,N}} c_S(m,n)!} \prod_{(p,q)\in S} a_{p,q}.
\end{align*}
Then, observe that the terms with powers of $a_{00}$ along with the integer coefficients can be taken out of the inner sum (of the double sum) and that the denominators of the multinomial coefficients are over nonzero indices in $\mathcal{K}_{M,N}$. This yields that the above equals:
\begin{align*}
&\sum_{\ell = 1}^{\min(\eta,i+j)} \sum_{\substack{S\in \mathcal{P}_{\eta}(i,j):\\c_S(0,0)=\eta-\ell}} \frac{\eta!}{(\eta-\ell)!\prod_{(m,n)\in \mathcal{K}_{M,N}^*} c_S(m,n)!} a_{00}^{\eta-\ell} \prod_{(0,0)\neq (p,q)\in S} a_{p,q}\\
= &\sum_{\ell = 1}^{\min(\eta,i+j)} \frac{\eta!}{(\eta-\ell)!} a_{00}^{\eta-\ell}
\sum_{\substack{S\in \mathcal{P}_{\eta}(i,j):\\c_S(0,0)=\eta-\ell}} \frac{1}{\prod_{(m,n)\in \mathcal{K}_{M,N}^*} c_S(m,n)!} \prod_{(0,0)\neq (p,q)\in S} a_{p,q}.
\end{align*}
The final step is to write the first term inside the outer sum as an appropriate derivatives at $a_{00}$, update the index for the inner sum to $\mathcal{P}_{\ell}^*(i,j)$, and appropriately update the product index as well. This yields:
\begin{align*}
f_{\jury}(A)_{ij}& = \sum_{\ell = 1}^{\min(\eta,i+j)} f^{(\ell)}(a_{00}) \sum_{S\in \mathcal{P}_{\ell}^*(i,j)} \frac{1}{\prod_{(m,n)\in \mathcal{K}_{M,N}^*} c_S(m,n)!} \prod_{(p,q)\in S} a_{p,q}\\
& = \sum_{\ell = 1}^{i+j} f^{(\ell)}(a_{00}) \sum_{S\in \mathcal{P}_{\ell}^*(i,j)} \frac{1}{\prod_{(m,n)\in \mathcal{K}_{M,N}^*} c_S(m,n)!} \prod_{(p,q)\in S} a_{p,q},
\end{align*}
where $f^{(\ell)}$ denotes the $\ell$-th order derivative of $f$. This computation, combined with \eqref{defn:p-of-A-for-jury} and the linearity of the derivatives, completes the required proof.

\item Suppose $\kappa\in [1:M+N-2]$ is the smallest such that $f(z)=(z-a_{00})^{\kappa}$ is the minimal polynomial. From the proof of Lemma~\ref{lemma-AB-prod-1} one can deduce that for all $(i,j)\in \mathcal{K}_{M,N}$ such that $i+j<\kappa$ we have $\big{(}f_{\jury}(A)\big{)}_{ij}=0$. On the other hand, from Theorem~\ref{T:poly-on-mat}$(a)$, we must have
\begin{align*}
\big{(}f_{\jury}(A)\big{)}_{ij}= \kappa!\sum_{S\in \mathcal{P}_{\kappa}^*(i,j)} \frac{1}{\prod_{(m,n)\in \mathcal{K}_{M,N}^*} c_S(m,n)!} \prod_{(p,q)\in S} a_{p,q} = 0,
\end{align*}
for all $(i,j)\in \mathcal{K}_{M,N}$ with $i+j\geq \kappa$. This completes the proof.

Finally, that no such $\kappa$ exists, and so $(z-a_{00})^{M+N-1}$ is the minimal polynomial, is Theorem~\ref{main-thm-1}.\qedhere
\end{enumerate}
\end{proof}

\begin{remark}[Definitions~\ref{jury-transform-1} and \ref{jury-transform-2}, and extension to fields]
We record two observations:
\begin{enumerate}
    \item In addition to formulating minimal polynomials, Theorem~\ref{T:poly-on-mat}$(a)$ fulfills another purpose noted earlier -- namely, it ensures that transforms in Definitions~\ref{jury-transform-1} and \ref{jury-transform-2} are compatible with the convolution of matrices.
    \item The discussion of minimal polynomials extends verbatim to any field of characteristic zero in place of~$\C$, thus extending the note in Remark~\ref{Rem:ext-to-rings}.
\end{enumerate}
\end{remark}

Let us demonstrate our findings on minimal polynomials with a few examples revisiting the $2\times 2$ case that we saw in the beginning of this subsection.

\begin{example} \
\begin{enumerate}[$(i)$]
    \item If $z-a_{00}$ is the required minimal polynomial, then all entries $a_{ij}$ for $(i,j)\neq (0,0)$ must be zero. And the converse also holds.

    \item For $(z-a_{00})^{2}$ to be the minimal polynomial, we must have $a_{ij}\neq 0$ for some $(i,j)\neq (0,0)$, and moreover must have,
\begin{align*}
    \sum_{S\in \mathcal{P}_{2}^*(i,j)} \frac{1}{\prod_{(m,n)\in \mathcal{K}_{M,N}^*} c_S(m,n)!} \prod_{(p,q)\in S} a_{p,q} = 0
\end{align*}
for all $(i,j)$ with $i+j\geq 2$. These conditions for $A:=(a_{ij})\in \C^{2\times 2}$ mean that one of $a_{0,1},a_{1,0},a_{1,1}$ is nonzero, and that $a_{0,1}a_{1,0}=0$. Therefore, either exactly one of $a_{0,1}$ and $a_{1,0}$ is zero, or $0=a_{0,1}=a_{1,0}\neq a_{1,1};$ precisely what we saw in Remark~\ref{Cayley-Hamilton-Jury-2by2-lemma-remark-2}.
    \item If $(z-a_{00})^{3}$ is minimal then we must have
    \begin{align*}
    \sum_{S\in \mathcal{P}_{3}^*(i,j)} \frac{1}{\prod_{(m,n)\in \mathcal{K}_{M,N}^*} c_S(m,n)!} \prod_{(p,q)\in S} a_{p,q} = 0
    \end{align*}
    for all $(i,j)\in \mathcal{K}_{M,N}$ with $i+j\geq 3$, but there exists some $(i,j)\in \mathcal{K}_{M,N}$ with $i+j\geq 2$ such that
    \begin{align*}
    \sum_{S\in \mathcal{P}_{2}^*(i,j)} \frac{1}{\prod_{(m,n)\in \mathcal{K}_{M,N}^*} c_S(m,n)!} \prod_{(p,q)\in S} a_{p,q} \neq 0
    \end{align*}
These conditions for $A\in \C^{2\times 2}$ mean that $a_{1,0}a_{0,1}\neq 0$, precisely what we have obtained in Remark~\ref{Cayley-Hamilton-Jury-2by2-lemma-remark-2}.
\end{enumerate}
\end{example}

\subsection{Jury's product versus the standard and Schur products. II}\label{R:CH-for-unbound-mat}

The Cayley--Hamilton property for the Jury product depends critically on the finiteness of the matrix \emph{dimension}, rather than just on the finiteness of the number of nonzero entries in a matrix. Let us explain.

Consider the union/direct limit of all finite dimensional matrices:
\[
\C^{\N\times \N}_{00} := \bigcup_{M,N=1}^{\infty}\C^{M\times N}.
\]
This space can be interpreted as the collection of semi-infinite--by--semi-infinite matrices with finitely many nonzero entries, indexed by
$
\mathcal{K}_{\infty, \infty} := [0:\infty)^2.
$
It admits the usual matrix product, and the Schur product, and the Jury product with properties similar to \eqref{E:jury-property} through \eqref{E:jury-inverse-2}. (The latter two products, in fact, extend naturally to all of $\C^{\N\times \N}$.) Thus, it is natural to ask whether Cayley--Hamilton-type theorems hold in this setting (as in Question~\ref{question}(1)). One knows that for every $A \in \C^{\N\times \N}_{00}$, there exist polynomials $p_1, p_2\in \C[z]$ such that
\begin{align*}
p_1(A) &= 0~ \mbox{ (for the standard product)}, \\
\mbox{and}\quad p_2[A] &= 0~ \mbox{ (for the Schur product)}.
\end{align*}
In fact, one can ensure that $p_1$ and $p_2$ are of the forms in equations~\eqref{eq:ch-std} and \eqref{eq:ch-schur} respectively. To the contrary, however, no analogous result holds for the Jury product:
\begin{prop}
Suppose $A=(a_{ij}) \in \C^{\N \times \N}_{00} \setminus \{{\bf 0}_{\N\times \N}\}$ is entrywise nonnegative and not a scalar multiple of the Jury product identity. Then for any integer $\kappa\geq 1$, the polynomial
\[
\mbox{$p(z):=(z-a_{00})^{\kappa}$ does not annihilate $A$.}
\]
\end{prop}
\begin{proof}
Suppose $(i,j)\in \mathcal{K}_{\infty,\infty}^*$ such that $a_{ij}>0$. Then for $(\kappa i,\kappa j)\in \mathcal{K}_{\infty,\infty}^*$, using the corresponding versions of \eqref{E:poly-act-prop} and Lemma~\ref{k-power-entries:thm}(2) for $\C_{00}^{\N\times \N}$, we have
\[
\big{(}(A-a_{00}\I_{\jury})^{\jury\kappa}\big{)}_{(\kappa i,\kappa j)} \geq a_{ij}^{\kappa} >0.
\]
This shows that $p_{\jury}(A)$ (computed according to the corresponding version of Theorem~\ref{T:poly-on-mat}$(a)$ for $\C_{00}^{\N\times \N}$) is nonzero, and completes the proof.
\end{proof}

In addition to the preceding observations, an even stronger phenomenon occurs. Unlike the rings induced by the standard and entrywise products on $\C_{00}^{\N\times\N}$, we have the following situation: for well-chosen $A \in \C_{00}^{\N\times \N}$ the space,
\[
\mathrm{span}_{\C}\{A\}\subseteq (\C_{00}^{\N\times \N},\jury,+)\quad \mbox{need \textit{not} be finite dimensional.}
\]
This leads to the following result, which further supports our claim:
\begin{prop}
There exists $A\in \C^{\N\times \N}_{00}$ such that $p_{\jury}(A)\neq 0$ for all nonzero $p\in \C[z]$.
\end{prop}
\begin{proof}
Take $A=\diag(0,1,0,0,\dots)\in \C_{00}^{\N\times \N}$. Then, using the corresponding version of Lemma~\ref{k-power-entries:thm}(2), for all integers $n\geq 1$:
\[
A^{\jury n} = \diag(0,\dots,0,1,0,\dots)\in \C_{00}^{\N\times \N},
\]
with $1$ at the $(n+1)$-th position. Now, if $p(z)=c_{0}+c_1z+\dots+c_nz^{n}\in \C[z]$ is nonzero, then
\[
p_{\jury}(A)=\mathrm{diag}(c_0,c_1,\dots,c_n,0,0,\dots) \in \C_{00}^{\N\times \N}.
\]
Since, $p$ is nonzero, we complete the proof.
\end{proof}

Therefore, one sees that in contrast with the standard and entrywise products, a Cayley–Hamilton theory over $\C_{00}^{\N\times \N}$ fails to hold for convolution. This observation clarifies the opening remark of the subsection -- showing that the answer to Question~\ref{question}(1) is ``no'' for $\C_{00}^{\N\times \N}$ -- and may also be regarded as a continuation of Subsection~\ref{Sub:jury-v-schur}. This concludes our discussion on the Cayley--Hamilton theory, and we move to the proofs of the positivity preservers.

\section{Positivity preservers: proofs}\label{S:proof-pos-preservers}

We adopt a systematic approach. We begin with establishing that some of the divided differences are nonnegative functions. We then specialize to various cases and derive nonnegativity, monotonicity, and continuity of the scalar functions. Combining continuity with the nonnegativity of divided differences, via a classical result of Boas and Widder \cite{boas1940functions} for continuous functions with nonnegative divided differences, we obtain that derivatives of the scalar function exist and are nonnegativity. Then via an application of the Bernstein theorem on absolutely monotonic functions\cite{bernstein1929sur}, we obtain real analyticity. This yields absolute monotonicity for our Schoenberg-type theorem.

\begin{defn}\label{defn:mid-con}
For an interval $I\subseteq \R$, we call $f:I\to \R$ nonnegative, nondecreasing, and convex, provided for all $y \geq x\in I$, we respectively have
\[
f(x)\geq 0,\quad f(y)\geq f(x), \quad (1-\lambda)f(x)+\lambda f(y)\geq f\left ((1-\lambda)x+\lambda y \right),
\]
for all $\lambda \in [0,1]$. Moreover, if the final inequality holds (only) for $\lambda=1/2$, then $f$ is called mid-convex.
\end{defn}

\begin{lemma}[Nonnegativity of difference operators]\label{L:diff-operators}
Suppose $I=(0,\rho)$ is an interval for some $0< \rho \leq \infty$. Let $\Omega \subseteq \C$ such that $\Omega\cap (0,\infty) = I$, and let $f:\Omega\to \C$ be a map. Fix an integer $N\geq 2$, and suppose $f_{\jury}(A)_{h}\in \Pn_N$  for all $A=(a_{ij})\in \Pn_N(\Omega)$, for all $h>0$ such that $(a_{00},a_{00}+2(N-1)h)\subseteq I$. Then the forward difference operators are nonnegative:
\[
(\Delta_{h}^\ell f)(x)\geq 0,\quad \mbox{for all } h>0,
\]
such that $x,x+h,\dots,x+2(N-1)h\in I$, for all $\ell = 0,1,\dots,N-1$.
\end{lemma}
\begin{proof}
Fix $x\in I$, and suppose $h>0$ is such that $x,x+h,\dots,x+2(N-1)h\in I$. Consider the following $N\times N$ matrix for all small $\epsilon>0$:
\[
A=A(\epsilon,x):=\diag(x-\epsilon,x-\epsilon,0,\dots,0)+\epsilon \mathbf{1}_{N\times N},
\]
where $\mathbf{1}_{N\times N}$ is the matrix with all entries 1. Clearly $A\in \Pn_N(I)$ for small $\epsilon>0$. Therefore $f_{\jury}(A)_{h}$ is positive semidefinite. Therefore, in particular, the diagonal entries of it are nonnegative:
\[
\mbox{for } k=1,\dots,N-1, \quad \mbox{we have } (f_{\jury}(A)_{h})_{k,k}= \frac{1}{k!}~x^{k}~(D_h^{k}f)(x)+g(\epsilon)\geq 0,
\]
for all small $\epsilon>0$. It is not difficult to see that $g(\epsilon)=o(1)$. Therefore, taking $\epsilon\to 0$ in the above relation, we obtain $(D_h^{k}f)(x)\geq 0$ for $k=1,\dots,N-1$. This yields the nonnegativity of the corresponding difference operators. For the $k=0$ case, observe that for the aforementioned $A:=A(\epsilon,x)$, we have $
(\Delta_{h}^{0}f)(x)=(D_h^0f)(x)=f(x)=(f_{\jury}(A)_h)_{00}\geq 0$, and thus we conclude the proof.
\end{proof}

\subsection{Proof of Theorem~\ref{T:polya-szego-for-jury}}

\begin{proof}[Proof of Theorem~\ref{T:polya-szego-for-jury}]
Pick any $N\geq 1$ and a positive semidefinite $A\in I^{N\times N}$. Suppose
\[
p_{n}(x):=\sum_{k=0}^{n}c_kx^k \quad \mbox{for all}\quad n\geq 0.
\]
Then $(p_{n})_{\jury}(A)$ is positive semidefinite. Since the positive semidefinite matrices are closed under taking entrywise limits, the matrix formed by entries $\lim_{n\to \infty} \big{(}(p_{n})_{\jury}(A)\big{)}_{ij}$ is positive semidefinite, if the limits exist. In the domain of convergence of $f$, we have pointwise convergence, which yields $\lim_{n\to\infty}p_{n}(a_{00})=f(a_{00})$, and for all $(i,j)\in \mathcal{K}_{N,N}^*$,
\begin{align*}
&\lim_{n\to \infty}\big{(}(p_{n})_{\jury}(A)\big{)}_{ij}\\
= &\sum_{\ell=1}^{i+j} \lim_{n\to \infty} p_{n}^{(\ell)}(a_{00})
\sum_{S\in\mathcal{P}_{\ell}^*(i,j)}
\frac{1}{\prod_{(m,n)\in \mathcal{K}_{M,N}^*} c_S(m,n)!}
\prod_{(p,q)\in S} a_{p,q} \\
=&\sum_{\ell=1}^{i+j} f^{(\ell)}(a_{00})
\sum_{S\in\mathcal{P}_{\ell}^*(i,j)}
\frac{1}{\prod_{(m,n)\in \mathcal{K}_{M,N}^*} c_S(m,n)!}
\prod_{(p,q)\in S} a_{p,q}=\big{(}f_{\jury}(A)\big{)}_{ij}.
\end{align*}
Therefore, we have that the following (which also gives the meaning of $f_{\jury}(A)$ in terms of the power series of $f(x)$):
\[
\sum_{k=0}^{\infty}c_k A^{\jury k}:=\lim_{n\to\infty}\sum_{k=0}^{n}c_k A^{\jury k}= \lim_{n\to \infty }(p_{n})_{\jury}(A)=f_{\jury}(A)
\]
exists and is positive semidefinite, completing the proof.
\end{proof}

\subsection{Proof of Theorem~\ref{T:Rudin-for-jury}}

\begin{proof}[Proof of Theorem~\ref{T:Rudin-for-jury}]
We begin with the proof of monotonicity.\medskip

\noindent\textit{(Monotonicity and nonnegativity).} To convey the idea of the proof in a simpler manner, let us first demonstrate it when $N=2$ and $0<\rho<\infty$. Fix $x\in I$ and suppose $h>0$ such that $x,x+h,x+2h\in I$. Under this assumption, using Lemma~\ref{L:diff-operators}, we have that the first forward difference operator is nonnegative:
\begin{align}\label{E0:L:N=2}
(\Delta^{1}_{h}f)(x)\geq 0 \implies f(x+h) \geq f(x),
\end{align}
if $h>0$, and $x,x+h,x+2h \in I$. In other words, for $x<y\in I$, and step size $h=(y-x)/2$, we have,
\begin{align}\label{E1:L:N=2}
f(x)\leq f\left (\frac{x+y}{2}\right) \quad \mbox{for all } x,y\in I.
\end{align}
To show that $f(x)\leq f(y)$ for arbitrary $x,y\in I$, we build on \eqref{E0:L:N=2} and \eqref{E1:L:N=2}. We begin with constructing a sequence converging to $y$. Given $x<y\in I$, we inductively define:
\[
y_0:=x,\quad y_{1} = \frac{x+y}{2}, \quad y_{2} = \frac{y_1+y}{2},\quad\dots,\quad y_n=\frac{y_{n-1}+y}{2},
\]
for all $n\geq 1$. Equation~\eqref{E1:L:N=2} shows that for $x<y\in I$,
\[
f(x)\leq f(\mbox{the mid-point of }x,y).
\]
Moreover, we have that each $y_n$ is the mid-point of $y_{n-1}<y\in I$. Combining these, we get
\begin{align}\label{E2:L:N=2}
f(x)\leq f(y_1) \leq f(y_2)\leq \dots \leq f(y_{n-1})\leq f(y_n) \quad \mbox{for all}\quad n\geq 1.
\end{align}
Moreover, we have $y_{n}\to y$ as $n\to\infty$. Therefore,
\(y-y_{n_0} < (\rho - y)/2\) for a large integer \(n_0\geq 2.\) Using this $n_0$, define $z=y+y-y_{n_0} = 2y-y_{n_0}$. Then we have $z\in I$, and $y$ as the mid-point of $y_{n_0}$ and $z$, where $y_{n_0}<z$. Therefore by \eqref{E1:L:N=2} for $x'=y_{n_0}$ and $y'=z$, we obtain
\[
f(y_{n_0})\leq f\left (\frac{y_{n_0}+z}{2}\right) = f(y).
\]
This and \eqref{E2:L:N=2} implies $f(x)\leq f(y)$ for $x<y\in I$. This completes the presentation for $N=2$.\medskip

Let us now look at the general $N\geq 2$ case with $0<\rho<\infty$. Here we have a similar approach, but with a bit of bookkeeping. To begin with, since $N\geq 2$, Lemma~\ref{L:diff-operators} yields that the first difference operator is nonnegative:
\begin{align*}
(\Delta^{1}_{h}f)(x)\geq 0 \quad \mbox{i.e.,}\quad f(x+h) \geq f(x),
\end{align*}
whenever $h>0$, and $x,x+h,\dots x+2(N-1)h \in I$. Therefore, given $x<y\in I$, for step size $h=\frac{y-x}{2(N-1)}$, we obtain that
\begin{align}\label{eq1:L:N=2}
x+h=\frac{(2N-3)x+y}{2(N-1)} \implies f(x)\leq f\left (\frac{(2N-3)x+y}{2(N-1)}\right),
\end{align}
for all $x<y\in I$. We now construct a sequence converging to $y$. Inductively define:
\[
y_0:=x, \quad  y_1:=\frac{(2N-3)y_0+y}{2(N-1)}, \quad \dots, \quad y_n:=\frac{(2N-3)y_{n-1}+y}{2(N-1)},
\]
for all $n\geq 1$. Therefore, equation~\eqref{eq1:L:N=2} shows that
\begin{align}\label{eq2:L:N=2}
f(x) \leq f(y_1)\leq  \cdots\leq f(y_{n-1})\leq f(y_n) \quad \forall~ n\geq 1.
\end{align}
Moreover, note that $y_{n}\to y$ as $n\to\infty$. Therefore as $\rho$ is finite,
\[
y-y_{n_0} < \frac{\rho-y}{2N-2} \quad \mbox{for a large integer}\quad n_0\geq 2.
\]
Now, using this $n_0$, consider $z\in \R$ defined by:
\[
z:=y+(2N-3)(y-y_{n_0}) \implies z < y+ \frac{2N-3}{2N-2}(\rho - y) < y + \rho - y = \rho.
\]
Therefore $z\in I$, and moreover $y=\frac{(2N-3)y_{n_0}+z}{2(N-1)}$ is the ``first step'' for step size $h=\frac{z-y_{n_0}}{2(N-1)}$ between $y_{n_0}<z$. Therefore by \eqref{eq1:L:N=2}, $f(y_{n_0})\leq f(y)$. This combined with \eqref{eq2:L:N=2} implies $f(x)\leq f(y)$ for all $x<y\in I$, showing the required monotonicity.

In the case when $\rho=\infty$, i.e., $I=(0,\infty)$, Lemma~\ref{L:diff-operators} shows that $(\Delta_{h}^{1}f)(x)\geq 0$ for all $x,h>0$. Therefore, for any $x<y\in I$, with step size $h={y-x}$, we obtain that $f(x)\leq f(y)$.

The nonnegativity of $f$ follows from Lemma~\ref{L:diff-operators} as $f(x)=(\Delta_h^{0}f)(x)\geq 0$ for all $x\in I$.\medskip

\noindent \textit{(Continuity).} We partition this proof into two cases. If $N\geq 3$, then the proof is rather simpler, and we begin with this case first (and treat the $N=2$ case after).\medskip

Since $N\geq 3$, for any $x\in I$ with $h>0$ such that $x,x+h,\dots, x+2(N-1)h\in I$, using Lemma~\ref{L:diff-operators}, we have that the second difference operator is nonnegative, and therefore:
\begin{align}\label{E0:T:cont}
(\Delta^{2}_{h}f)(x)\geq 0 \implies \frac{f(x)+f(x+2h)}{2}\geq  f(x+h),
\end{align}
whenever $h>0$ such that $x,x+h,\dots, x+2(N-1)h\in I$.

On the other side, from Theorem~\ref{T:Rudin-for-jury} on monotonicity, we have that $f$ is nondecreasing over $I$. This means that $f:I\to\R$ can have at most jump discontinuities, and that the right- and left-hand limits exist for all $x\in I$ (since $I$ is open). Define the right- and left-hand limits:
\[
f^+,f^-:I\to \R \quad \mbox{defined by}\quad f^{\pm}(x):=\lim_{y \to x^{\pm}} f(y) \quad \mbox{for all } x\in I.
\]
By monotonicity: $f^-(x)\leq f(x)\leq f^+(x)$ for all $x\in I$. Moreover, by \eqref{E0:T:cont} for small $h>0$:
\begin{align*}
&\lim_{h\to0^+}\frac{f(x)+f(x+2h)}{2}\geq  \lim_{h\to0^+} f(x+h) \implies \frac{f(x)+f^+(x)}{2}\geq  f^+(x).
\end{align*}
This implies $ f(x)\geq f^+(x)$, and therefore, $f(x)=f^+(x)$. Thus, $f$ can have jumps only from the left side.

Again, for any $x\in I$ and small $h>0$, we have the second difference operator nonnegative:
\begin{align*}
(\Delta^{2}_{h}f)(x-h)\geq 0 &\implies \frac{f(x-h)+f(x+h)}{2}\geq  f(x).
\end{align*}
Now take limit, and use $f(x)=f^+(x)$, to obtain:
\begin{align*}
\lim_{h\to 0^+} \frac{f(x-h)+f(x+h)}{2}\geq f(x) &\implies  \frac{f^-(x)+f^+(x)}{2}\geq  f(x).
\end{align*}
This implies $f^-(x)\geq f(x)$, and therefore $f(x) = f^-(x)$, showing that $f$ is continuous.\medskip

So, all that remains is to show the $N=2$ case. From the previous $N\geq 3$ case we import: $f$ has at most jump discontinuities, and $f^{\pm}$ are well defined over $I$ with $f^-(x)\leq f(x)\leq f^+(x)$. Therefore, all we need is to prove $f^-(x)\geq f(x)\geq f^+(x)$ for all $x\in I$.\medskip

Suppose $A=\begin{pmatrix} a & b \\ \overline{b} & c
\end{pmatrix} \in \Pn_2(\Omega)$. Then for $h>0$ such that $a,a+h,a+2h\in I$, we have
\[
f_{\jury}(A)_h=
\begin{pmatrix}
f(a) & b (D^{1}_hf)(a) \\
\overline{b} (D^{1}_hf)(a) & c(D^{1}_hf)(a) + |b|^2 (D^{2}_hf)(a)
\end{pmatrix} \in \Pn_2.
\]
Expanding the determinant of the above, we get
\begin{align*}
&\det f_{\jury}(A)_{h} = cf(a)(D^{1}_hf)(a) + |b|^2 f(a) (D^{2}_hf)(a) - |b|^2 \big{(}(D^{1}_hf)(a)\big{)}^2 \\
&= cf(a)\frac{f(a+h)-f(a)}{h}
+ |b|^2 f(a) \frac{f(a)-2f(a+h)+f(a+2h)}{h^2} \\
&\qquad - |b|^2 \left(\frac{f(a+h)-f(a)}{h}\right)^2.
\end{align*}
This determinant is nonnegative for all small $h>0$. Therefore, multiplying $h^2$ on both sides, and taking limit we get,
\begin{align*}
&\lim_{h\to 0^+}h^2\det f_{\jury}(A)_{h}\\
&= |b|^2 f(a) \Big{[}f(a)-2f^+(a)+f^+(a)\Big{]} - |b|^2 \Big{[}f^+(a)-f(a)\Big{]}^2 \geq 0.
\end{align*}
Assuming $b\neq 0$, this simplifies to,
\begin{align*}
0\leq &  f(a) \Big{[}f(a)-2f^+(a)+f^+(a)\Big{]} - \Big{[}f^+(a)-f(a)\Big{]}^2 \\
=& f(a) \Big{[}f(a)-f^+(a)\Big{]} - \Big{[}f(a)-f^+(a)\Big{]}^2 \\
=& \Big{[}f(a)-f^+(a)\Big{]} \Big(f(a)  - f(a)+f^+(a)\Big) = \Big{[}f(a)-f^+(a)\Big{]} f^+(a).
\end{align*}
If $f^+(a)\neq 0$ then $f(a)\geq f^+(a)$, which implies that $f(a)= f^+(a)$. On the other hand, if $f^+(a)=0$ then
\[
0\leq f(a)\leq f^{+}(a)=0 \implies f(a)=f^+(a).
\]
Therefore, as one can always construct $A=(a_{ij})\in \Pn_2(\Omega)$ with $a_{00}=a$ and $a_{0,1}=b\neq 0$, for arbitrary $a\in I$, we have $f(a)=f^+(a)$ for all $a\in I$.

Next, suppose $A=\begin{pmatrix} a & b \\ \overline{b} & c
\end{pmatrix} \in \Pn_2(\Omega)$ is such that
\[
B(h):=\begin{pmatrix} a-h & b \\ \overline{b} & c-h
\end{pmatrix}\in \Pn_2(\Omega) \quad \mbox{for all small } h>0.
\]
Then for these small $h>0$, we have
\begin{align*}
&f_{\jury}(B(h))_h\\
&=
\begin{pmatrix}
f(a-h) & b (D^{1}_hf)(a-h) \\
\overline{b} (D^{1}_hf)(a-h) & (c-h)(D^{1}_hf)(a-h) + |b|^2 (D^{2}_hf)(a-h)
\end{pmatrix}
\end{align*}
is positive semidefinite. Expanding its determinant we get,
\begin{align*}
&\det f_{\jury}(B(h))_{h}\\
&= (c-h)f(a-h)(D^{1}_hf)(a-h) + |b|^2 f(a-h) (D^{2}_hf)(a-h)\\
&\qquad - |b|^2 \big{(}(D^{1}_hf)(a-h)\big{)}^2 \\
&= cf(a-h)\frac{f(a)-f(a-h)}{h} \\
& \qquad + |b|^2 f(a-h) \frac{f(a-h)-2f(a)+f(a+h)}{h^2} \\
& \qquad \qquad - |b|^2 \left(\frac{f(a)-f(a-h)}{h}\right)^2.
\end{align*}
The above determinant is nonnegative for all small $h>0$. Therefore, multiplying $h^2$ on both sides, and taking limit, we get
\begin{align*}
&\lim_{h\to 0^+}h^2\det f_{\jury}(B(h))_{h}\\
&= |b|^2 f^-(a) \Big{[}f^-(a)-2f(a)+f^+(a)\Big{]} - |b|^2 \Big{[}f(a)-f^-(a)\Big{]}^2 \geq 0.
\end{align*}
Assuming $b\neq 0$, the above, combined with $f(a)=f^+(a)$, implies
\begin{align*}
& f^-(a) \Big{[}f^-(a)-f(a)\Big{]} - \Big{[}f^{-}(a)-f(a)\Big{]}^2 \\
&= \Big{[}f^-(a)-f(a)\Big{]} \Big( f^-(a)  - f^{-}(a)+f(a)\Big{)} = \Big{[}f^-(a)-f(a)\Big{]} f(a) \geq 0.
\end{align*}
Suppose $f^{-}(a) \geq 0$ for all $a\in I$. Then, $f^-(a) > 0$ implies $f(a)>0$, and therefore $f^-(a)\geq f(a)$. On the other hand, if $f(a)=0$, then $0\leq f^{-}(a)\leq f(a)=0$, which implies $f^{-}(a)=f(a)$.

So, all that remains to show is that $f^-$ is nonnegative on $I$, and matrices $A=(a_{ij})\in \Pn_2(\Omega)$ exists for any arbitrary $a_{00}=a\in I$ with nonzero $a_{01}=b$, such that $B(h)\in \Pn_2(\Omega)$. We show both claims in one step. Consider
\[
A=\begin{pmatrix} a & \sqrt{a/2} \\ \sqrt{a/2} & a
\end{pmatrix}, \quad \mbox{and so}\quad B(h)=\begin{pmatrix} a-h & \sqrt{a/2} \\ \sqrt{a/2} & a-h
\end{pmatrix}\in \Pn_2(\Omega)
\]
for all small $h>0$. Now for the first problem, look at the first entry of $f_{\jury}(B(h))_h$: it is $f(a-h)\geq 0$. Now since entrywise limit of positive matrices is positive, we have $\lim_{h\to 0^+}f_{\jury}(B(h))_h \in \Pn_2$. Therefore the first entry of the limit, given by $\lim_{h\to 0^+}f(a-h) = f^-(a)$ is nonnegative. Since $a\in I$ is arbitrary, we are done.\medskip

\noindent\textit{(Convexity).} Since $\rho=\infty$ and the arbitrary $N\geq 3$, we have that $\Delta^2_h(x)\geq 0$ for all $x,h>0$. Therefore, we take $h={(y-x)}/2$ for any $y>x\in I$, and use that whenever $x,h>0$, we have
\begin{align*}
\Delta^{2}_{h}(x)\geq 0 &\implies \frac{f(x)+f(x+2h)}{2}\geq  f(x+h)\\
&\implies \frac{f(x)+f(y)}{2}\geq  f\left(\frac{x+y}{2}\right).
\end{align*}
This yields mid-convexity. Therefore, using {\cite[Chapter~6, Lemma~6.3]{khare2022matrix}}, we get that $f$ is rationally convex: convex for all rational $\lambda \in [0,1]$. Thus, by continuity (proved above) we get convexity.
\end{proof}

\begin{remark}[Comparing with entrywise preservers, and more]

A couple of notes are in order.
\begin{enumerate}
    \item The proofs for monotonicity and continuity of functions $f:I\to \R$ acting as entrywise positivity preservers $f[-]:\Pn_N(I)\to \Pn_N$ proceed along a different path. In the entrywise setting, monotonicity follows almost immediately, while continuity is derived via \textit{multiplicative} mid-convexity of $f$ combined with the monotonicity (via {\cite[Chapter~6, Theorem~6.2]{khare2022matrix}}). In contrast, for the transforms $f_{\jury}(-)_{h}$, the argument for monotonicity is naturally more involved, and continuity is proved directly without appealing to the stronger notion of mid-convexity. However, we also get (mid-)convexity for certain unbounded domains.

    \item As noted earlier in Remark~\ref{rem:jury-transform-2}, we may have used the positivity of $f_{\jury}(-)_h$ for all possible $h>0$ in our proof above: first in the proof of monotonicity $h>0$ can be large, and then in the proof of continuity $h>0$ can be small.
\end{enumerate}
\end{remark}

\subsection{Proof of Theorem~\ref{T:horn-for-jury}}

\begin{proof}[Proof of Theorem~\ref{T:horn-for-jury}]
We begin with the proof of the first assertion.
\begin{enumerate}[$(a)$]
    \item We know from Theorem~\ref{T:Rudin-for-jury} that $f$ is continuous. Furthermore, we know from Lemma~\ref{L:diff-operators} that all forward differences $(\Delta_{h}^{\ell}f)(x)\geq 0$ for all $x,h>0$, for all $\ell=0,1,\dots,N-1$. Therefore, using the theorem of Boas and Widder on continuous functions with positive differences (see \cite{boas1940functions} or \cite[Theorem 13.8]{khare2022matrix}), we conclude that $f\in C^{N-3}(I)$. Moreover, we get that $f^{(\ell)}(x)\geq 0$ for all $x\in I$ and all $\ell = 0,1,\dots,N-3$. Using the same result, we get the rest of the characterization.
    \item Suppose $x\in I=(0,\rho)$ for $0<\rho\leq \infty$. Then for small $\epsilon>0,$
\[
A=A(\epsilon,x):=\diag(x,x,0,\dots,0)+\epsilon \mathbf{1}_{N\times N}\in \Pn_{N}(I),
\]
where $\mathbf{1}_{N\times N}$ is the matrix with all entries $1$. Then $f_{\jury}(A)_{00}=f(x+\epsilon)\geq 0$; therefore by continuity $f(x)\geq 0.$ We also have, for $k=1,\dots,N-1,$ that
\[
f_{\jury}(A)_{k,k}=\frac{1}{k!}(x+\epsilon)^{k}~f^{(k)}(x+\epsilon)+g(\epsilon)\geq 0 \quad \mbox{for small }\epsilon>0,
\]
where $g(\epsilon)=o(1)$. Taking $\epsilon\to 0$, we obtain $f^{(k)}(x)\geq 0$ for $k\in 1,\dots,N-1.$\qedhere
\end{enumerate}
\end{proof}

\subsection{Proof of Theorem~\ref{T:schoeberg-for-jury}}

\begin{proof}[Proof of Theorem~\ref{T:schoeberg-for-jury}] $(a)\implies(b)$ for $\rho=\infty:$ Under the assumption in Theorem~\ref{T:schoeberg-for-jury}$(a)$, using Theorem~\ref{T:Rudin-for-jury} we conclude that $f:I\to \R$ is continuous. Moreover, using Lemma~\ref{L:diff-operators}, its forward difference operators of order $\ell$ are nonnegative for all $\ell \geq 2$. Therefore, using the theorem of Boas and Widder on continuous functions with nonnegative forward differences (see \cite{boas1940functions} or \cite[Theorem 13.8]{khare2022matrix}), we conclude that $f\in C^{\ell-2}(I)$ for all $\ell \geq 2$. In particular $f\in C^{\infty}(I)$. Finally, as the set of positive semidefinite matrices are closed, the entrywise limit $f_{\jury}(A) = \lim_{h\to 0^+}f_{\jury}(A)_h$ shows the remaining assertion in $(b)$.

$(b)\implies(c):$ Suppose $g:=f|_{I}$. Then $g_{\jury}(-):\Pn_N(I)\to \Pn_N$ for all $N\geq 1$. Therefore by Theorem~\ref{T:horn-for-jury}, all derivatives of \( g \) satisfy \( g^{(k)}(x) \geq 0 \) for all \( x \in I \). This implies that \( g \) is nonnegative and nondecreasing on \( I \), and thus admits a continuous extension to the origin via the limit \( g(0) := \lim_{x \to 0^+} g(x) \geq 0 \).

By the Bernstein theorem for absolutely monotonic functions (see \cite{bernstein1929sur} or \cite[Theorem~14.3]{khare2022matrix}), it follows that \( g \) coincides with a convergent power series \( \sum_{k=0}^\infty c_k z^k \) on the disc \( D(0, \rho) \) centered at $0$ of radius $\rho$. Restricting this power series representation of \( g \) to \( I \), and since \( g^{(k)}(x) \geq 0 \) on \( I \) for all \( k \), we deduce that \( g^{(k)}(0) \geq 0 \), and hence \( c_k \geq 0 \) for all \( k \geq 0 \), completing the argument.

$(c)\implies(b):$ This follows from Theorem~\ref{T:polya-szego-for-jury}.

$(c)\implies(d):$ Fix an arbitrary $N\geq 1$ and a positive semidefinite $A\in \Omega^{N\times N}$. We have from Theorem~\ref{T:polya-szego-for-jury} that $f_{\jury}(A)=\sum_{k=0}^{\infty}c_k A^{\jury k}$ is positive semidefinite.

Now, if $f$ is a linear function, then $f_{\jury}(A)_{h}=f_{\jury}(A)$, which is positive semidefinite. If $f$ is not linear, then using Theorem~\ref{T:jury} that convolution of positive definite matrices (positive semidefinite and nonsigular) is positive definite, we obtain that $f_{\jury}(A+\delta\I)$ is positive definite, for all $\delta>0$, where $\I$ is the identity matrix for the usual matrix product. We have that the set of positive definite matrices is open in the set of Hermitian matrices, and that $\lim_{h\to 0^+}f_{\jury}(A+\delta\I)_{h}=f_{\jury}(A+\delta\I)$. Therefore we must have for all small $h> 0$ that $f_{\jury}(A+\delta\I)_{h}$ is positive semidefinite. Now take $\delta\to 0^+$ to conclude the implication.
\end{proof}

\begin{example}\label{Example:T-Schoenberg}
It is not difficult to see that for linear absolutely monotonic $f:I\to \R$, each $(D^{\ell}_{h}f)(x)=f^{(\ell)}(x)$. Therefore, for these maps $(c)\implies (a)$ in Theorem~\ref{T:schoeberg-for-jury}.

So, for a counterexample, consider $f(x)\equiv x^2$. Then $(D_h^1f)(x)=2x+h$ and $(D_h^2f)(x)=2$. It is not difficult to verify that for well-chosen $A\in \Pn_2(\R)$, $f_{\jury}(A)_{h}\not\in \Pn_2(\R)$ for large $h>0$.
\end{example}

\subsection{Proof of Theorem~\ref{T:fh-for-jury}}

\begin{proof}[Proof of Theorem~\ref{T:fh-for-jury}] We begin with the proof of the first assertion.
\begin{enumerate}[$(a)$]
    \item If $\alpha<0$ then $f'(x)=\alpha x^{\alpha-1}<0$ for some $x\in I$, contradicting Theorem~\ref{T:horn-for-jury}. Therefore $\alpha\geq 0.$ Next we show that this is sufficient. Suppose $A:=\begin{pmatrix}a & b \\ b & c\end{pmatrix}\in \Pn_2(I)$. Then
\begin{align*}
    f_{\jury}(A)&=\begin{pmatrix}f(a) & bf'(a) \\ bf'(a) & cf'(a)+b^2f''(a)\end{pmatrix}\\
    &=\begin{pmatrix}a^{\alpha} & \alpha ba^{\alpha-1} \\ \alpha {b} a^{\alpha-1} & \alpha c a^{\alpha-1}+\alpha(\alpha-1)b^2 a^{\alpha-2}\end{pmatrix}.
\end{align*}
We need: $\det f_{\jury}(A)\geq 0$ and $f_{\jury}(A)_{1,1}\geq 0$ (as $f_{\jury}(A)_{00}=a^{\alpha}\geq 0$ already). We have both:
\begin{align*}
\alpha(ac-b^2) + \alpha^2 b^2\geq 0  &\implies  \alpha ac + \alpha(\alpha-1) b^2\geq 0 \\
&\implies f_{\jury}(A)_{1,1}= \alpha a^{\alpha-1}c + \alpha(\alpha-1) b^2a^{\alpha-2}\geq 0,
\end{align*}
and now the determinant,
\begin{align*}
 ac\geq b^2 &\implies \alpha ac + \alpha(\alpha - 1)b^2 - \alpha^2b^2\geq 0\\
&\implies \alpha a^{2\alpha-1}c + \alpha(\alpha - 1)b^2a^{2\alpha-2} - \alpha^2b^2 a^{2\alpha-2}=\det f_{\jury}(A)\geq 0.
\end{align*}
\item This follows from Theorem~\ref{T:horn-for-jury}$(b)$ in the following way. Let $\alpha<N-2$ be not a nonnegative integer. If $\alpha<0$ then $f'(x)=\alpha x^{\alpha-1}<0$ for $x>0$, a contradiction. Similarly, if $\alpha\in (k-2,k-1)$, where $k\leq N-1$ is an integer, then
\[
f^{(k)}(x)= \alpha (\alpha-1)\dots(\alpha-k+1)x^{\alpha-k}.
\]
Since $k-2<\alpha<k-1$, the above derivative is negative for $x>0$, which is a contradition to Theorem~\ref{T:horn-for-jury}$(b)$. Therefore, $f(x)\equiv x^\alpha$ is not a positivity preserver. Thus, there must exist a positive semidefinite $A$ such that $f_{\jury}(A)$ is not so.\qedhere
\end{enumerate}
\end{proof}

\section{Concluding remarks}\label{S:concluding}

\begin{remark}[The $\rho=\infty$ cases in Theorems~\ref{T:schoeberg-for-jury} and \ref{T:horn-for-jury}]
Although one might hope to dispense with the assumption $\rho=\infty$ in the implication $(a)\!\implies\!(b)$ of Theorem~\ref{T:schoeberg-for-jury} and Theorem~\ref{T:horn-for-jury}(a), our proofs indicate that this condition is not easily removed. The difficulty lies in reconciling a classical theorem of Boas and Widder \cite{boas1940functions} on divided differences with the $h$-dependent transforms $f_{\jury}(-)_h$. Nevertheless, it would be of interest to investigate the versions over bounded domains.
\end{remark}

Recall Theorem~\ref{T:fh-for-jury}$(b)$: it presents with an open case of $N\geq 3$. Looking at the simpler though ``convoluted'' proof for the $N=2$ case, it may not be straightforward to follow a similar quantitative path to show that the required determinants are nonnegative. In light of this, we propose a possible simplification and formulate a (future) question via an entrywise transform and a differential operator.

\begin{quest}\label{future-question}
Suppose $I=(0,\rho)$ for $0<\rho\leq \infty$, and fix an integer $N \geq 3$. For $\alpha > N-2$ not an integer, define
\begin{align}\label{eq:drvt-rel}
\mathcal{B}(\alpha,A)_{ij}:= \left.
\frac{\partial^{i+j}}{\partial x^{i}\, \partial y^{j}}
\Bigg( a_{00} + \sum_{(m,n)\in \mathcal{K}_{N,N}^*} a_{m,n} x^{m} y^{n} \Bigg)^{\alpha}
\;\right|_{\substack{x=0 \\ y=0}}.
\end{align}
Is it true that $\mathcal{B}(\alpha,A):=\big{(}\mathcal{B}(\alpha,A)_{ij}\big{)}_{i,j=0}^{N-1}\in \Pn_N$, for all $A\in \Pn_N(I)$?
\end{quest}

The resolution of Question~\ref{future-question} yields the open case in Theorem~\ref{T:fh-for-jury}$(b)$ via the following qualitative formulation of the transform.

\begin{prop}\label{Prop:for-conj}
Consider the premise of Question~\ref{future-question}, and suppose $f(x)\equiv x^\alpha$. Then for all $A\in \Pn_N(I)$,
\[
\diag\big(1,1,{2!},\dots,{(N-1)!}\big)~f_{\jury}(A)~\diag\big(1,1,{2!},\dots,{(N-1)!}\big)=\mathcal{B}(\alpha,A).
\]
\end{prop}
\begin{proof}
Fix $A=(a_{ij})\in \Pn_N(I)$, define the following for small $|x|,|y|>0$:
\[
F(x,y):=a_{00}+G(x,y):= a_{00}+\sum_{(m,n)\in\mathcal{K}_{N,N}^*} a_{m,n} x^m y^n.
\]
We get the following using the binomial series for small $|x|,|y|>0$,
\begin{align*}
F(x,y)^\alpha
= \sum_{k=0}^\infty \frac{1}{k!}f^{(k)}(a_{00}) G(x,y)^k.
\end{align*}
Applying the multinomial theorem for $k\geq 1$ we get,
\begin{align*}
&G(x,y)^k= \sum_{(i,j)\in \mathcal{K}_{\infty,\infty}^*} x^i y^j\sum_{S\in \mathcal{P}_{k}^*(i,j)} \frac{k!}{\prod_{(m,n)\in \mathcal{K}_{N,N}^*} c_S(m,n)!} \prod_{(p,q)\in S} a_{p,q}.
\end{align*}
Note, the above sum is finite: $(i,j)\in \mathcal{K}_{\infty,\infty}\setminus \mathcal{K}_{k(N-1)+1,k(N-1)+1}$ implies $\mathcal{P}_{k}^{*}(i,j)=\emptyset$.
For convenience, we follow a convention that the sum over $\mathcal{P}_{0}^*(0,0)$ is 1, and the sum over $\mathcal{P}_{0}^*(i,j)$ is 0 for nonzero $(i,j)$. (This can be seen as $\mathcal{P}_{0}^*(0,0)$ being True and $\mathcal{P}_{0}^*(i,j)$ being False.) Canceling the $k!$ terms, we get
\begin{align*}
&F(x,y)^\alpha \\
= & \sum_{k=0}^\infty f^{(k)}(a_{00}) \sum_{(i,j)\in \mathcal{K}_{\infty,\infty}} x^i y^j\sum_{S\in \mathcal{P}_{k}^*(i,j)} \frac{1}{\prod_{(m,n)\in \mathcal{K}_{N,N}^*} c_S(m,n)!} \prod_{(p,q)\in S} a_{p,q}\\
=&\sum_{k=0}^\infty  \sum_{(i,j)\in \mathcal{K}_{\infty,\infty}} x^i y^j \bigg{[} f^{(k)}(a_{00})  \sum_{S\in \mathcal{P}_{k}^*(i,j)} \frac{1}{\prod_{(m,n)\in \mathcal{K}_{N,N}^*} c_S(m,n)!} \prod_{(p,q)\in S} a_{p,q}\bigg{]}.
\end{align*}
Therefore, the coefficient of $x^i y^j$ in $F(x,y)^\alpha$ is given by the following, and combined with Definition~\ref{jury-transform-1}, we obtain the desired form:
\begin{align*}
&\frac{1}{i! j!}\left.
\frac{\partial^{i+j}}{\partial x^{i}\, \partial y^{j}} F(x,y)^{\alpha}
\;\right|_{\substack{x=0 \\ y=0}}\\
=& \sum_{k=0}^\infty  f^{(k)}(a_{00})  \sum_{S\in \mathcal{P}_{k}^*(i,j)} \frac{1}{\prod_{(m,n)\in \mathcal{K}_{N,N}^*} c_S(m,n)!} \prod_{(p,q)\in S} a_{p,q}\\
= & \sum_{k=0}^{i+j}  f^{(k)}(a_{00})  \sum_{S\in \mathcal{P}_{k}^*(i,j)} \frac{1}{\prod_{(m,n)\in \mathcal{K}_{N,N}^*} c_S(m,n)!} \prod_{(p,q)\in S} a_{p,q} = f_{\jury}(A)_{ij}.
\end{align*}
Here, for $(i,j)\in \mathcal{K}_{N,N}^*$, the outer sum truncates at $i+j$ because $\mathcal{P}_{k}^*(i,j)$ is empty, and hence the sum over $\mathcal{P}_{k}^*(i,j)$ is zero, for all $k>i+j$. This concludes the proof.
\end{proof}

Finally, the above qualitative formulation for power functions can be generalized for any sufficiently differentiable function using the multivariate Faà di Bruno's formula. As this extension is not needed for our purpose here, we do not go in that direction. This concludes the paper.

\section*{Acknowledgements} JM was supported by the Canada Research Chairs program (CRC-2022-00097) and the Fulbright Foundation. JM and MN were supported by NSERC Discovery grant (Canada) (RGPIN-2024-04232). PKV was supported by the Centre de recherches math\'ematiques and Universit\'e Laval (CRM-Laval) Postdoctoral Fellowship and the Alliance grant (ALLRP 577214-2022).

% \bibliographystyle{plain}
% \bibliography{biblio4}

{\footnotesize
\setlength{\bibsep}{2pt plus 0.3pt}
\bibliographystyle{plain}
\bibliography{biblio}
}
\end{document}